\documentclass{amsart}
\usepackage{amssymb,amsmath}

\usepackage{fancyhdr} 

\usepackage{amsthm}

\usepackage{color}
\definecolor{red-}{rgb}{1.0,0.0,0.0}
\definecolor{grey}{rgb}{0.6, 0.6, 0.6}
\definecolor{brown}{rgb}{0.5,0.2,0.0}
\definecolor{brown-}{rgb}{0.0,0.1,1.0}
\definecolor{green-}{rgb}{0.0, 0.6, 0.0}
\definecolor{gold}{rgb}{0.8,0.7,0.0}
\definecolor{black}{rgb}{0.0,0.0,0.0}
\definecolor{DarkGreen}{rgb}{0.0,0.3,0.2}
\definecolor{LightGreen}{rgb}{0.8,1.0, 0.8}
\definecolor{yellow}{rgb}{0.9,0.9,0.0}

\theoremstyle{plain}% default
\allowdisplaybreaks                          % for proof reading
\newcommand{\E}{{\mathcal E}}
\newcommand{\F}{{\mathcal F}}

\newcommand{\LL}{{\mathcal L}}
\newcommand{\Ol}{{\mathcal O}}
\newcommand{\T}{{\mathcal T}}
\newcommand{\U}{{\mathcal U}}
\newcommand{\V}{{\mathcal V}}

\newcommand{\proj}{\mathbb P}
\newcommand{\pu}{{\mathbb P^1}}
\newcommand{\pd}{{\mathbb P^2}}
\newcommand{\pt}{{\mathbb P^3}}

\newcommand{\pn}{{\mathbb P^n}}

\newcommand{\FF}{\mathbb F}

\newtheorem{theorem}{Theorem}[section]
\newtheorem{lemma}[theorem]{Lemma}
\newtheorem{proposition}[theorem]{Proposition}
\newtheorem{corollary}[theorem]{Corollary}
\theoremstyle{definition}

\newtheorem{example}[theorem]{Example}

\theoremstyle{remark}
\newtheorem{rmk}[theorem]{Remark}

\def\thetitle{Triple solids and scrolls}
\begin{document}

\author{Antonio Lanteri}

\author{Carla Novelli}

\address{Dipartimento di Matematica ``Federigo Enriques'', \newline
\indent
Universit\`a degli Studi di Milano, \newline
\indent
via C. Saldini, 50, \newline
\indent
I-20133 Milano, Italy
%%\newline fax: 0039-02-50316090
}
\email{antonio.lanteri@unimi.it}

\address{Dipartimento di Matematica ``Giuseppe Peano", \newline
\indent
Universit\`a degli Studi di Torino, \newline
\indent
via Carlo Alberto, 10, \newline
\indent
I-10123 Torino, Italy}
\email{carla.novelli@unito.it}

\subjclass[2010]
%[2020]
{Primary:14C20,14E20,14J60;\ Secondary:14J30.
}

\title{Triple solids and scrolls}

\begin{abstract} Let $Y$ be a smooth projective variety of dimension $n \geq 2$ endowed
with a finite morphism  $\phi:Y \to \mathbb P^n$ of degree $3$, and suppose that $Y$, polarized by some ample line bundle,
is a scroll over a smooth variety $X$ of dimension $m$. Then $n \leq 3$ and either $m=1$ or $2$.
When $m=1$, a complete description of the few varieties $Y$ satisfying these conditions is provided.
When $m=2$, various restrictions are discussed showing that in several instances the
possibilities for such a $Y$ reduce to the single case of the Segre product $\mathbb P^2 \times \mathbb P^1$.
This happens, in particular, if $Y$ is a Fano threefold as well as if the base surface $X$ is $\mathbb P^2$.
\end{abstract}

\dedicatory{dedicated to Enrique Arrondo on the occasion of his 60th birthday}

%\date{September 13, 2023}

\maketitle

%%%%%%%%%%%%%%%%%%%%%%%%%%%%%%%%%%%%%%%%%%%%%%%%%%%%%%%%%%%%%%%%%%%%%%%%%%%%%%
%%%%%%%%%%%%%%%%%%%                SECTION 0             %%%%%%%%%%%%%%%%%%%%%
%%%%%%%%%%%%%%%%%%%%%%%%%%%%%%%%%%%%%%%%%%%%%%%%%%%%%%%%%%%%%%%%%%%%%%%%%%%%%%

\section*{Introduction} \label{intro}

As observed by Fujita in his book \cite[(10.9.1), (10.11)]{libroFujita}, and reaffirmed in the supplementary
note circulating as a manuscript ``Problems on Polarized Varieties'', the problem of classifying
the triple covers of $\mathbb P^n$
is still open, in particular for $n=2$ and $3$, in which cases such coverings are not
necessarily of triple section type \cite{Fu0}.
Moreover, as far as we know, there has been no recent contribution
in this direction, except for \cite{FPV}, where the authors consider a special class of surfaces
represented as triple planes. Actually, in spite of several general results on triple
covers existing in the literature (e.g.\ see \cite{Mi3ple}, \cite{T}),
nothing seems explicitly aimed at the study of $3$-dimensional triple solids and, more specifically, of
threefolds which admit at the same time a projective bundle structure.
The unsolved case left open in \cite[Sec.\ 4]{LN3} exactly reflects this lack of knowledge.

More generally, let $Y$ be a projective $n$-fold and let $\phi:Y \to \mathbb P^n$ be a finite morphism of degree $d$.
Suppose that $d=2$ or $3$ and that $Y$ is a scroll
for some polarization at the same time. Then, a classical result
of Lazarsfeld \cite{Laz1} implies that $n=2$ if $d=2$ and that $n=2$ or $3$ if  $d=3$. In this paper we investigate
the possible varieties occurring precisely in this setting.
In Section \ref{Examples} we present some concrete examples
to illustrate this situation.
The hearth of the paper
is Section \ref{classification} in which we provide their description, which is complete as far as scrolls over a curve
are concerned. The crucial remark is that if $Y$ is a scroll with respect to some polarizing line bundle, then it is also
a scroll, via the same projection, with respect to the ample and spanned line bundle $H:=\phi^*\mathcal O_{\mathbb P^n}(1)$.
This allows us to work with the ample and spanned vector bundle obtained by pushing down $H$
via the scroll projection.
According to this,
the case of double solids is easily settled by Proposition \ref{double}, while the more delicate
case of triple solids that are scrolls over a curve is dealt with in
Theorem \ref{maintheorem}. The last three Sections
are devoted to discussing the case of triple solids that are scrolls over a surface, which
looks even more interesting and intriguing.
More specifically, in Section \ref{on surfaces}
we exhibit several different situations
in which the only possibility for $Y$ is given by the Segre product $\mathbb P^2 \times \mathbb P^1$;
in particular, this is the case if we assume in addition that $Y$ is a Fano manifold.  Furthermore,
we have the opportunity to amend a flaw in a result of Ballico \cite[Theorem]{Ba}.
Moreover,
in Section \ref{constraints} we provide further restrictions on $Y$ deriving from the consideration of
the triple plane induced  by $\phi$ on the general element $S \in \phi^*|\mathcal O_{\mathbb P^3}(1)|$
and of its Tschirnhaus bundle. Finally, in Section \ref{overP2} we focus on scrolls
whose base surface is $\mathbb P^2$. Proposition \ref{P2} shows the extremely severe conditions
that this hypothesis entails on the various numerical characters involved in the discussion.
Comparing the Tschirnhaus bundle of $\phi$ with that of the triple plane induced on
$S$ in a special situation arising from our analysis, we finally succeed to prove that necessarily
$Y=\mathbb P^2 \times \mathbb P^1$, even in this case (Theorem \ref{conclusion}).

However, from a complementary point of view suggested by this result,
we would like to emphasize that also for scrolls $Y$ over $\mathbb P^2$
with respect to some ample line bundle $L$, distinct from the product,
it may happen that $Y$ contains a smooth surface $S$ with the structure of a triple plane even
as a very ample divisor, but with $\mathcal O_Y(S) \not= L$
(see Remark \ref{final}).

Throughout the whole paper a relevant role is played by Miranda's formulas, which allow
to express the invariants of a triple plane by means of the Chern classes of its Tschirnhaus bundle.

%%%%%%%%%%%%%%%%%%%%%%%%%%%%%%%%%%%%%%%%%%%%%%%%%%%%%%%%%%%%%%%%%%%%%%%%%%%%%%
%%%%%%%%%%%%%%%%%%%                SECTION 1             %%%%%%%%%%%%%%%%%%%%%
%%%%%%%%%%%%%%%%%%%%%%%%%%%%%%%%%%%%%%%%%%%%%%%%%%%%%%%%%%%%%%%%%%%%%%%%%%%%%%

\section{Background material} \label{background}

\setcounter{equation}{0}

We work over the field of complex numbers and we use the standard notation from algebraic geometry.
By a little abuse we make no distinction between a line bundle and the corresponding invertible sheaf.
Moreover, the tensor products of line bundles are denoted additively.
The pullback $i^\ast \E$ of a vector bundle $\E$ on $X$ by an embedding of projective varieties
$i\colon Y \hookrightarrow X$ is denoted by $\E_Y$.
We denote by $K_X$ the canonical bundle of a smooth variety $X$.

A {\em polarized manifold} is a pair $(X,\LL)$ consisting of a smooth projective variety $X$ and an ample line bundle $\LL$ on $X$.
The sectional genus and the $\Delta$-genus of a polarized manifold $(X,\LL)$ are defined as
$g(X,\LL)=1+\frac{1}{2}\big(K_X+(\dim X -1)\LL\big)\cdot \LL^{\dim X -1}$
and $\Delta(X,\LL)=\dim X+\LL^{\dim X}-h^0(X,\LL)$, respectively.
A polarized manifold $(X,\LL)$ is said to be a {\em scroll} (over
$W$) if it is a classical scroll, namely if there exist a smooth projective variety $W$ of positive
dimension and a surjective morphism $\pi\colon X \to W$ such that
$(F,\LL_F)\cong\big(\proj^m,\Ol_{\proj^m}(1)\big)$ with $m=\dim X - \dim
W$ for any fiber $F$ of $\pi$. This condition is equivalent to
saying that $X = \proj_W(\F)$ for some ample
vector bundle $\F$ of rank $\geq 2$ on $W$, and $\LL$ is
the tautological line bundle.
If $\LL$ is a line bundle on a
projective manifold $X$, we denote by $\varphi_\LL$ the rational
map $X --\to \proj^{h^0(\LL)-1}$ associated with the complete linear
system $|\LL|$.

We will use the symbol $\mathbb F_e$ to denote the Segre--Hirzebruch surface $\mathbb P\big(\mathcal O_{\mathbb P^1} \oplus \mathcal O_{\mathbb P^1}(-e)\big)$
of invariant $e (\geq 0)$, and $\sigma$ and $f$
will denote the section of minimal self-intersection $-e$ and a fiber respectively.

%%%%%%%%%%%%%%%%%%%%%%%%%%%%%%%%%%%%%%%%%%%%%%%%%%%%%%%%%%%%%%%%%%%%%%%%%%%%%%
%%%%%%%%%%%%%%%%%%%                SECTION 2             %%%%%%%%%%%%%%%%%%%%%
%%%%%%%%%%%%%%%%%%%%%%%%%%%%%%%%%%%%%%%%%%%%%%%%%%%%%%%%%%%%%%%%%%%%%%%%%%%%%%

\section{Some covers of $\pn$ admitting a scroll structure} \label{Examples}
\setcounter{equation}{0}

Let $Y$ be a smooth projective variety with $\dim Y = n$ and let $d \geq 2$ be
an integer: if $Y$ is endowed with a
$d$-uple branched covering of $\mathbb P^n$ we will refer to $Y$ as a \text{\it{ $d$-uple $n$-solid}}.
Any smooth projective variety $Y$ of dimension $n$ embedded in some projective space
can be regarded as a $d$-uple $n$-solid, where $d=\deg Y$, by projecting it onto a $\mathbb P^n$
from a suitable linear space.
This is true in particular when $Y$ is a scroll over a positive dimensional
projective variety. In this case, however, the
integers $n$ and $d$ are not completely unrelated, due to the following general fact.

\begin{lemma}\label{basic}
Let $Y$ be any projective bundle
over a smooth positive dimensional
projective variety and a $d$-uple $n$-solid at the
same time. Then $d \geq n \geq 2$.
\end{lemma}
\proof Clearly $n \geq 2$. Suppose that $d \leq n-1$. Since $Y$ has Picard number
$\rho(Y) \geq 2$, we get a contradiction by a well known result of Lazarsfeld (see
\cite[Proposition 3.1]{Laz1}). \qed

\begin{example}\label{rational} Let $Y=\proj_\pu(\V)$, where
\begin{equation}\label{normal} \V=\Ol_\pu\oplus
\Ol_\pu(\alpha_1)\oplus \dots \oplus\Ol_\pu(\alpha_{n-1})
\end{equation}
with $0\leq \alpha_1\leq\dots\leq \alpha_{n-1}$ and let
$\alpha=\sum_{i=1}^{n-1} \alpha_i=\deg\V$. Let $\xi$ be the
tautological line bundle, let $F$ be a fiber of the projection
$\pi\colon Y \to \pu$ and let $L=\xi+bF$ for some integer $b$. Due
to the normalization \eqref{normal}, we know that $L$ is ample if
and only if it is very ample if and only if $b>0$ (\cite[Lemma
3.2.4]{BS}). So let $b>0$; then the morphism $\varphi_L$
embeds $Y$ in $\proj^N$ as a scroll of degree
\begin{equation}\label{grado}
d:=L^n=\deg\big(\V\otimes\Ol_\pu(b)\big)=\alpha+nb
\end{equation}
by the Chern--Wu relation, where
$N+1=h^0(L)=h^0\big(\V\otimes\Ol_\pu(b)\big)=n+nb+\alpha=n+d$. Let
$\Lambda$ be a general linear subspace of $\proj^N$ of dimension $N-1-n$.
Then, projecting $\varphi_L(Y)$ from $\Lambda$ onto a $\pn\subset
\proj^N$ skew with $\Lambda$, we get a map $\phi\colon Y \to \pn$,
which is a finite morphism of degree $d$ and $L=\phi^\ast\Ol_{\pn}(1)$. On the other hand,
$(Y,L)$ is a scroll over $\pu$.
Clearly $d \geq n$, according to Lemma \ref{basic}. In this
specific case this simply follows from \eqref{grado}, taking into
account that $b>0$ and $\alpha \geq 0$.
\end{example}
\setcounter{equation}{0}

\begin{rmk}\label{ris1}
If $(Y,L)$ is
an $n$-dimensional scroll as in Example \ref{rational} and $d=n$, then
$(Y,L)=\big(\pu\times\proj^{n-1},
\Ol_{\pu\times\proj^{n-1}}(1,1)\big)$. Actually, equality $d=n$
implies $b=1$ and $\alpha=0$ by \eqref{grado}, and the latter in turn
implies that $\V=\Ol_\pu^{\oplus n}$.
\end{rmk}

In particular, according to Example \ref{rational}, we get: for
$(n,d)=(2,2)$ the smooth quadric surface of $\pt$ described as a
double plane via projection from a general point; for
$(n,d)=(2,3)$ the rational cubic scroll of $\proj^4$ described as
a triple plane via projection from a general line; for
$(n,d)=(3,3)$ the Segre product $\pu\times \pd\subset \proj^5$
described as a triple solid via projection from a general line.
In all these cases, of course, the line bundle $L$ making $Y$ a scroll is very ample.

\begin{example}\label{irrational}
Let $B \subset \mathbb{P}^2$ be an irreducible projective curve
whose dual $B^{\vee} \subset \mathbb{P}^{2 \vee}$ is a smooth
curve of degree $d$. The following construction is inspired by
\cite[\S 3]{Cat}. In $\mathbb P^2 \times \mathbb P^{2 \vee}$
consider the incidence variety $T= \{(x,\ell) \in \mathbb P^2 \times
\mathbb P^{2 \vee}\ |\ x \in \ell\}$. Then $T$ is a smooth threefold
which is endowed with two $\mathbb{P}^1$-bundle structures
$p\colon T \to \mathbb P^2$, $q\colon T \to \mathbb P^{2 \vee}$
via the projections of $\mathbb P^2 \times \mathbb P^{2 \vee}$
onto the factors.

Now consider the smooth curve $B^{\vee}$ and let $S =
q^{-1}(B^{\vee})$. Clearly $S$ is a smooth surface and
$\pi:=q|_S\colon S \to B^{\vee}$ makes
$S$ a $\pu$-bundle over $B^\vee$. On the other hand, since $S$ is
a divisor on $T$ belonging to the linear system $|q^*\mathcal
O_{\mathbb P^{2 \vee}}(d)|$, we see that $f:=p|_{S}\colon S \to
\mathbb P^2$ is a finite morphism of degree $\deg f = p{^{-1}}(x)
\cdot S$ where $x \in \mathbb{P}^2$ is any point. We thus get
\begin{equation}
\deg f = \big(p^\ast\mathcal
O_{\mathbb P^2}(1)\big)^2 \cdot q^\ast\mathcal O_{\mathbb P^{2
\vee}}(d) = d. \nonumber
\end{equation}
Looking at the construction more closely, we can
note that the branch locus of the $d$-uple plane $f\colon S \to \mathbb
P^2$ is $B$. To see this, note that $S = \{(x,\ell)\ |\ x\in \ell,
\ell \in B^{\vee}\}$, while, $p^{-1}(x) = \{(x,\ell)\ |\ \ell \ni
x \} = \{(x,\ell)\ |\ \ell \in L_x \}$, for any $x \in \mathbb
P^2$, where $L_x$ is the line in $\mathbb P^{2 \vee}$
corresponding to the pencil of lines through $x$. Now fix $x \in
\mathbb P^2$: then
\begin{equation}
f{^{-1}}(x)= p^{-1}(x) \cap S = \{(x,\ell)\ |\ \ell \ni x,
\ell \in B^{\vee}\}= \{(x,\ell)\ |\ \ell \in L_x \cap
B^{\vee}\}. \nonumber
\end{equation}
This shows that the pre-images of $x$ via $f\colon S
\to \mathbb P^2$ correspond to the intersections of $L_x$ with
$B^{\vee}$. As $B^{\vee}$ is smooth those pre-images are not $d$
distinct points if and only if the line $L_x$ is tangent to
$B^{\vee}$. By biduality, this is equivalent to saying that $x \in
B$.

Let $L=f^\ast \Ol_\pd(1)$. Then $L$ is an ample and spanned line
bundle on $S$. Note that $L=\big(p^\ast \Ol_\pd(1)\big)_S$.
Moreover, for any fiber $\frak{f}$ of $\pi\colon S \to B^{\vee}$
we have $L \cdot\frak{f}=p^\ast \Ol_\pd(1)\cdot \big(q^\ast
\Ol_{\pd^\vee}(1)\big)^2=1$. This says that $(S,L)$ is a scroll
over $B^\vee$, a smooth curve of genus $\binom{d-1}{2}$.

Set $\LL=\mathcal{O}_{\mathbb P^2 \times \mathbb P^{2 \vee}}(1,1)$
and recall that $(\mathbb P^2 \times \mathbb P^{2 \vee},\LL)$ is
the del Pezzo fourfold of degree six. Since $T\in |\LL|$,
$(T,\LL_T)$ is the following del Pezzo threefold of degree six:
$T=\proj_{\pd^\vee}(T_{\pd^\vee})$, with $\LL_T$ being the
tautological line bundle \cite[Chapter I, \S8 (8.7), (8.8)]{libroFujita}. Furthermore,
\begin{equation}
\LL_S=(\LL_T)_S=\big(p^\ast\Ol_\pd(1)+q^\ast\Ol_{\pd^\vee}(1)\big)_S=
f^\ast\Ol_\pd(1)+\pi^\ast\Ol_{B^\vee}(1)=L+\pi^\ast\Ol_{B^\vee}(1). \nonumber
\end{equation}
This shows that
$S=\proj_{B^\vee}\big((T_{\pd^\vee}(-1))_{B^\vee}\big)$, with $L$
being the tautological line bundle.
\end{example}

In particular this gives
\begin{example}\label{ris2}
Set $d=3$. Then the previous construction exhibits a smooth
surface which is a triple plane and a scroll over an elliptic
curve at the same time. We can show that
$\big(T_{\pd^\vee}(-1)\big)_{B^\vee}=\U \otimes \Ol_{B^\vee}(z)$,
where $z\in {B^\vee}$, and $\U$ is an indecomposable vector bundle
of degree one on the elliptic curve ${B^\vee}$. Thus, letting
$M=L+\frak{f}_0$, where $\frak{f}_0$ is any fiber of $\pi$, we see
that $M$ is very ample \cite[Exerc.\ 2.12, p.\ 385]{Ha} and $\varphi_M$ embeds $S$ as an
elliptic quintic scroll in $\proj^4$. Since $L=M-\frak{f}_0$, we
can regard the triple plane $f\colon S \to \pd$ as the projection
of the quintic elliptic scroll from its fiber $\frak{f}_0$ onto a
plane skew with it. Note that here $L$ is ample and spanned but not very ample.
\end{example}

%%%%%%%%%%%%%%%%%%%%%%%%%%%%%%%%%%%%%%%%%%%%%%%%%%%%%%%%%%%%%%%%%%%%%%%%%%%%%%
%%%%%%%%%%%%%%%%%%%                SECTION 3             %%%%%%%%%%%%%%%%%%%%%
%%%%%%%%%%%%%%%%%%%%%%%%%%%%%%%%%%%%%%%%%%%%%%%%%%%%%%%%%%%%%%%%%%%%%%%%%%%%%%

\section{Double and triple solids
admitting a scroll structure} \label{classification}
\setcounter{equation}{0}

In this Section we slightly change the perspective.
Let $\phi: Y \to \mathbb P^n$ be a $d$-uple $n$-solid, where $d=2$ or $3$:
we wonder when $Y$ admits an ample line bundle $L$ such that $(Y,L)$
is a scroll over a projective manifold of dimension $m \geq 1$.
Let us consider the line bundle $H=\phi^*\mathcal O_{\mathbb P^n}(1)$,
which, of course, is ample and spanned (in principle, it could also be
very ample, as some examples in Section \ref{Examples} show).
Let $\pi:Y \to X$ be the projection of the scroll $(Y,L)$: since $\text{\rm{Pic}}(Y)$
is generated by $L$ and $\pi^*\text{\rm{Pic}}(X)$ we can write
$H = aL + \pi^*\mathcal O_X(D)$, where  $a$ is a positive integer
and  $D$ is a divisor on $X$. Then
\begin{eqnarray} \label{d=aK}
d &=& H^n = (aL + \pi^*D)^n  \\ \nonumber
&=& a^nL^n + n a^{n-1}L^{n-1} \cdot \pi^*D + \dots + \binom{n}{m} a^{n-m}L^{n-m} \cdot (\pi^*D)^m \\ \nonumber
&=& aK,  \nonumber
\end{eqnarray}
since $n \geq m+1$, where
\begin{equation} \label{K}
K = a^{n-1}L^n + na^{n-2}L^{n-1} \cdot \pi^*D + \dots + \binom{n}{m}a^{n-m-1}L^{n-m} \cdot (\pi^*D)^m.
\end{equation}
Now, if $d\ (\geq 2)$ is prime, we deduce from \eqref{d=aK} that either $a = 1$ (and then $K = d$), or $a = d$, which implies that
$1 = K = d^{n-m-1}K^{\prime}$, where $K^{\prime}$ is an integer, and this is impossible if $n \geq m + 2$. On the other hand,
if $a = 1$, then the pair $(Y, H)$ itself is a scroll over $X$, hence we can suppose that $Y = \mathbb P(\mathcal E)$,
where $\mathcal E = \pi_*H$ is a vector bundle on $X$ of rank $n - m + 1$, which is ample and spanned,
so being its tautological line bundle $H$. In particular, let $d=2$;
then $n=2$ by Lemma \ref{basic}, hence
$m = 1$, i.e. $X$ is a curve. Then necessarily $a = 1$;
otherwise we would get $a = 2$,
and then $1 = K = 2\big(L^2 + \deg(D)\big)$ by \eqref{K}, which is absurd.
Similarly, let $d=3$; then $n=2$ or $3$ by Lemma \ref{basic}, hence either $n=2$ and $m=1$, or $n=3$ and $m=1, 2$.
If $n = 3$, i.e. $Y$ is a threefold, then we obtain that necessarily $a=1$.
Otherwise, we would get $a=3$ and
\begin{equation} \nonumber
1 = K = \begin{cases}
9L^3+9L^2 \cdot \pi^*D  & \text{if  $m=1$},\\ \nonumber
9L^3+9L^2 \cdot \pi^*D + 3 L \cdot (\pi^*D)^2 & \text{if  $m=2$},\\
\end{cases}
\end{equation}
by \eqref{K},
which is clearly impossible. Therefore, $a=1$ in all these cases, hence the problem becomes determining
when $(Y,H)$ itself is a scroll.
Note that case $d=3$ with $n=2$ is not covered by the previous analysis; it will be
discussed in the proof of Theorem \ref{maintheorem}.
For $d=2$ the answer is very easy and is given by the following

\begin{proposition} \label{double}
Let $\phi: Y\to \mathbb P^n$ be any smooth double $n$-solid with $n \geq 2$.
Then there is no polarization on $Y$ making it a scroll over a smooth projective variety of positive dimension
except for $\big(Y, \phi^\ast \mathcal O_{\mathbb P^n}(1)\big) = \big(\pu\times\pu, \Ol_{\pu\times\pu}(1,1)\big)$.
\end{proposition}

\begin{proof}
Let $\phi\colon Y \to \proj^{n}$ be the finite morphism of degree $2$ making $Y$ a double
$n$-solid and consider the ample and spanned line bundle $H:=\phi^\ast \Ol_{\proj^{n}}(1)$.
If $\pi:Y \to X$
is a scroll for some polarization, then $n=2$, $X$ is a curve and
$(Y,H)$ itself is a scroll via $\pi$  in view of the above discussion.
Then $\mathcal E:=\pi_*H$ is an ample and spanned rank-$2$ vector bundle on $X$.
Let $B \in |\Ol_\pd(2b)|$, $b \geq 1$, be the branch locus of $\phi$. Comparing the expression of
\begin{equation}
K_Y=-2H+\pi^\ast(K_X+\det\E) \nonumber
\end{equation}
with that given by the ramification formula
\begin{equation}
K_Y=\phi^\ast \Ol_\pd(b-3)=(b-3)H \nonumber
\end{equation}
we conclude that $b=1$ and $K_X+\det\E=\Ol_X$, because $H$ and $\pi^*\mathcal O_X(1)$
are linearly independent in $\text{\rm{Pic}}(Y)$.
This shows that $X=\mathbb P^1$
and $\det \mathcal E = \mathcal O_{\mathbb P^1}(2)$, which in turn implies
$\mathcal E= \mathcal O_{\mathbb P^1}(1)^{\oplus 2}$. Equivalently,
$(Y,H)=\big(\pu\times\pu,\Ol_{\pu\times\pu}(1,1)\big)$.
\end{proof}

\setcounter{equation}{0}

As to case $d=3$ is concerned, here let us start assuming that $m=1$.
\begin{theorem}\label{maintheorem}
Let $\phi: Y\to \mathbb P^n$ be any smooth triple $n$-solid with $n \geq 2$.
Then there is no polarization on $Y$ making it a scroll over a smooth projective curve
except for the following three pairs $\big(Y, \phi^\ast \Ol_{\proj^{n}}(1)\big)$:
\begin{itemize}
\item[$(1)$] $\big(\pd\times\pu, \Ol_{\pd\times\pu}(1,1)\big)$;
\item[$(2)$] $(\FF_1,[\sigma+2f])$, where $\sigma$ is the $(-1)$-section and $f$ is a fiber;
\item[$(3)$] $\big(\proj_C(\U),L\big)$, where $\U$ is an indecomposable vector bundle of degree one
on an elliptic curve $C$ and $L$ is the tautological line bundle of $\U\otimes\Ol_C(z)$,
$z$ being a point of $C$.
\end{itemize}
\end{theorem}

\begin{proof}
Let $\phi\colon Y \to \proj^{n}$ be the finite morphism of degree $3$ making $Y$ a triple $n$-solid
and consider the ample and spanned line bundle $H:=\phi^\ast \Ol_{\proj^{n}}(1)$ again.
Assume that $(Y,L)$ is a scroll over a smooth curve $X$ of genus $q:=g(X)$ for some ample
line bundle $L$ and let $\pi\colon Y \to X$ be the scroll projection.
From Lemma \ref{basic} we know that $n=2$ or $3$.

First suppose that $n=3$. Then, according to the discussion at the beginning of this Section,
$(Y,H)$ itself is a scroll over $X$.
By \cite[Theorem 1]{Laz1} we know that $\phi$ induces an isomorphism
$0=H^1(\mathbb P^3,\mathbb C) \cong H^1(Y, \mathbb C)$. Therefore
$h^1(\mathcal O_Y)=0$, and then the scroll structure of $(Y,H)$ over $X$ implies
that $q=0$. Thus $(Y,H)$ is a scroll over $\mathbb P^1$, so
$Y=\proj_\pu\big(\oplus_{i=1}^3\Ol_\pu(a_i)\big)$.  Since $H$ is ample and $H^3=3$,
we derive $Y=\proj_\pu\big(\Ol_\pu(1)^{\oplus 3}\big)$, hence $(Y,H)=\big(\pd\times\pu, \Ol_{\pd\times\pu}(1,1)\big)$
(see Remark \ref{ris1}).

Next, let $n=2$. In this case what we observed at the beginning of this Section
implies that $a=1$, unless in the following case:
\begin{equation} \label{K=1}
a=3  \qquad   \text{\rm{and}}  \qquad  K = 1.
\end{equation}

{\it{Claim}.} Case \eqref{K=1} cannot occur.

To prove the claim, consider the scroll $(Y,L)$ again, set $\mathcal E^{\prime}=\pi_*L$, so that
$L^2 = \deg \mathcal E^{\prime}$,  and recall that $H=aL+\pi^*\mathcal O_X(D)$ for some divisor
$D$ on $X$. If \eqref{K=1} holds, then \eqref{K} gives
\begin{equation} \label{a3}
1 = K = 3 \deg \mathcal E^{\prime}+ 2 \deg D.
\end{equation}
Let's prove that \eqref{a3} does not occur.
We can write $\phi_\ast \Ol_Y = \Ol_\pd \oplus \T$, where $\T$, the Tschirnhaus bundle
of $\phi$, is a vector bundle of rank $2$ on $\pd$. Then the branch locus of $\phi$ is an
element of $|2 \det \T^{\vee}|$ \cite[Proposition 4.7]{Mi3ple}. Set $b_i = c_i(\T)$. By applying the Riemann--Hurwitz formula
to the curve $\phi^{-1}(\ell)$ where $\ell \subset \pd$ is a general line, we get
\begin{equation}\label{ga3}
2g(Y,H)-2 = 3(-2)+(-2b_1).
\end{equation}
On the other hand, since $K_Y = -2L + \pi^\ast(K_X+\det \E^{\prime})$, taking into account the
expression of $H$ and condition \eqref{a3}, the genus formula shows that
\begin{equation}
2g(Y,H)-2 = (K_Y+H) \cdot H = 2\ (3 \deg \E^{\prime} + 2 \deg D) + 6\ (q-1) = 6q -4. \nonumber
\end{equation}
Combining this with \eqref{ga3} we get
\begin{equation}\label{b1a3}
-b_1 = 3q+1.
\end{equation}
Now, since $Y$ is a $\pu$-bundle over $X$, we know that $K_Y^2=8(1-q)$ and the
topological Euler--Poincar\'e characteristic is $e(Y) = 4(1-q)$.
Thus, eliminating $b_2$ from Miranda's formulas for the triple plane $\phi:Y \to \mathbb P^2$ \cite[Proposition 10.3]{Mi3ple}
\begin{equation}\label{M form}
K_Y^2 = 27+12b_1+2b_1^2-3b_2 \qquad \text{and} \qquad e(Y) = 9 + 6b_1 + 4b_1^2 - 9b_2,
\end{equation}
and using \eqref{b1a3} we obtain the following equation $9q^2 - 29q + 12=0$,
which has no integral solution. This proves the claim.

Therefore $a=1$ even if $n=2$, hence $(Y,H)$ itself is a scroll over $X$; so
$g(Y,H)=q$ and then the Riemann--Hurwitz formula applied to the curve $\phi^{-1}(\ell)$ now gives
\begin{equation}\label{b1}
-b_1 = q+2,
\end{equation}
Moreover, $K_Y^2=8(1-q)$ and  $e(Y) = 4(1-q)$ again.  In this case, eliminating $b_2$ from
Miranda's formulas, we get $q(q-1)=0$. If $q=0$, from Example \ref{rational} we see that
$\alpha=b=1$, hence $Y=\mathbb P\big(\mathcal O_{\mathbb P^1}(1) \oplus \mathcal O_{\mathbb P^1}(2)\big)$:
this gives case (2) in the statement. On the other hand, for $q=1$ we get case (3).
This is a consequence of the following lemma.
\end{proof}

\setcounter{equation}{0}

\begin{lemma} \label{lemma}
Let $(Y,H)$  be a surface scroll over a smooth curve $C$ of genus one, for some ample and spanned
line bundle $H$. If $H^2=3$, then $Y=\mathbb P_C(\mathcal U)$, $\mathcal U$ being  the nontrivial
extension
\begin{equation}
0 \to \mathcal O_C \to \mathcal U \to \mathcal O_C(p) \to 0,  \nonumber
\end{equation}
with $p \in C$, and
$H=[\sigma + f]$, where $\sigma$ denotes the tautological section on $Y$.
\end{lemma}
\begin{proof}
Write $Y=\mathbb P_C(\mathcal V)$, where $\mathcal V$ is a rank-2 vector bundle on $C$,
that we can suppose to be normalized as in \cite[p.\ 373]{Ha}. Denote by $\sigma$ and $f$
the tautological section and a fiber, respectively. Then $\sigma^2=-e$, where $e=- \deg \mathcal V$
is the invariant of $Y$. Since $(Y,H)$ is a scroll, up to numerical equivalence, we can write
$H = [\sigma + bf]$ for some integer $b$. Thus $H^2=-e+2b$, and then condition
$H^2=3$ gives $b=\frac{1}{2}(e+3)$, which implies that $e$ is odd. Moreover,
the ampleness conditions say that $b > e$ if $e\geq 0$ and $b\geq 0$ if $e=-1$
\cite[Propositions 2.20 and 2.21, p.\ 382]{Ha}. This, combined with the above expression of
$b$ shows that there are only two possible cases, namely:
\begin{equation}\label{2poss}
(e,b) = (-1,1) \quad{\text{\rm{or}}} \quad (1,2).
\end{equation}
In the latter case, $H=[\sigma + 2f]$ is clearly not spanned, since its restriction to the elliptic curve
$\sigma$ has degree $\deg H_{\sigma}=(\sigma + 2f) \cdot \sigma =1$. On the contrary,
in the former case, $\mathcal V=\mathcal U$ \cite[pp.\ 376--377]{Ha}
and we can check the spannedness of $H$ by using Reider's theorem \cite{R}.
Set $M=H-K_Y$, then $M=3 \sigma$, up to numerical equivalence.
In particular $M^2=9 > 5$, hence Reider's theorem applies. Suppose, by contradiction,
that $H=K_Y+M$ is not spanned; then there exists an effective divisor $D$ on $Y$ such that either
\begin{equation} \label{caso1}
D \cdot M=0 \quad \text{\rm{and}} \quad D^2=-1,
\end{equation}
or
\begin{equation} \label{caso2}
D \cdot M=1 \quad \text{\rm{and}} \quad D^2=0.
\end{equation}
Up to numerical equivalence we can write $D=x \sigma + y f$ for suitable integers $x,y$,
and then we get $D \cdot M = 3(x+y)$, while $D^2=x(2y+x)$.
Clearly, the expression of $D \cdot M$ rules out the possibility in \eqref{caso2}.
Suppose \eqref{caso1} holds. Then
$x=1$ and $y=-1$. So $D=[\sigma - f]$ and therefore $D \cdot \sigma = 0$.
However, since $e=-1$,
the elliptic curve $\sigma$ moves in an algebraic family (parameterized by the base curve $C$ itself),
sweeping out the whole surface $Y$. Thus the equality $D \cdot \sigma=0$ would imply
that $D$ cannot be effective, a contradiction.
Therefore $H$ is spanned in the former case of \eqref{2poss}.
\end{proof}

%%%%%%%%%%%%%%%%%%%%%%%%%%%%%%%%%%%%%%%%%%%%%%%%%%%%%%%%%%%%%%%%%%%%%%%%%%%%%%
%%%%%%%%%%%%%%%%%%%                SECTION 4            %%%%%%%%%%%%%%%%%%%%%
%%%%%%%%%%%%%%%%%%%%%%%%%%%%%%%%%%%%%%%%%%%%%%%%%%%%%%%%%%%%%%%%%%%%%%%%%%%%%%

\section{Scrolls over surfaces} \label{on surfaces}
\setcounter{equation}{0}

Consider triple $n$-solids again. According to Lemma \ref{basic}, apart from scrolls over curves,
there is only one more possibility for $Y$ being a scroll for
some polarization, namely, that $n=3$ and $\dim X = 2$.  In this Section and
the following ones we focus precisely on this case,
showing that $Y$ must satisfy several restrictions. A further motivation for this study is provided by an unresolved situation
in \cite[p.\ 687]{LN3}.
So, let $\phi:Y \to \mathbb P^3$ be a triple solid, and suppose that $(Y,L)$ is a scroll over a smooth surface $X$ via $\pi:Y \to X$,
for some ample line bundle $L$.
In this case the argument at the beginning of Section 3 says that
\begin{equation}\label{setting}
(Y,H)\ \text{\rm{itself is scroll over}}\ X\ \text{\rm{via}}\ \pi,\ \text{\rm{where}}\ H=\phi^*\mathcal O_{\mathbb P^3}(1).
\end{equation}
We can thus suppose that $\mathcal E:=  \pi_* H$
is an ample and spanned rank-2 vector bundle on $X$ and $Y= \mathbb P_X(\mathcal E)$,
with tautological line bundle $H$. When we refer to \eqref{setting}, implicitly we also mean that $\mathcal E$ is as above.
In this case, since $\pi_* \mathcal O_Y= \mathcal O_X$ \cite[Proposition 7.11, p.\ 162]{Ha},
we have
\begin{equation} \label{irreg}
h^i(\mathcal O_Y)=h^i(\mathcal O_X) \quad i=0, \dots , 3
\end{equation}
\cite[Exerc.\ 4.1, p.\ 222]{Ha}. In particular, $\chi(\mathcal O_Y)=\chi(\mathcal O_X)$.

Notice that the pair $(Y,H)$ as in case (1) of Theorem \ref{maintheorem} can also be regarded as a scroll over a surface
by taking $(X,\mathcal E)=\big(\mathbb P^2, \mathcal O_{\mathbb P^2}(1)^{\oplus 2}\big)$.
We will refer to this case as the {\it{obvious case}} in the subsequent discussion.
First of all, given any smooth triple solid $\phi : Y \to \mathbb P^3$ and $H = \phi^*\mathcal O_{\mathbb P^3}(1)$,
we have $h^0(H) \geq 4$; on the other hand $H^3 = 3$, hence the $\Delta$-genus of $(Y, H)$ is
$\Delta (Y, H) = 6 - h^0(H) \leq 2$. In our setting (i.e. taking into account the additional scroll structure of $Y$),
the situation is simpler. In fact we have

\begin{proposition} \label{very ample}
Either $\Delta(Y, H) = 2$, or $\Delta(Y, H) = 0$ and $(Y, H)$ is as in the obvious case; in particular,
if $H$ is very ample, then $(Y, H)$ is as in the obvious case.
\end{proposition}

\begin{proof} Suppose that $\Delta(Y, H) < 2$; since $Y$ is a scroll over a surface, its Picard number is
$\rho(Y) \geq 2$, so combining this with Fujita's classification of polarized manifolds of low
$\Delta$-genus \cite[Theorem 5.10 and Corollary 6.7]{libroFujita} we immediately get what is stated.
In particular, if $H$ is very ample then $|H|$ embeds our threefold $Y$ in $\mathbb P^N$, with
$N=h^0(H)-1 \geq 4$, hence it cannot be $\Delta(Y,H)=2$.
\end{proof}

\setcounter{equation}{0}

\begin{rmk}\label{Ballico}
In particular, the fact that in the setting \eqref{setting} it can be $\Delta(Y,H)=2$ amends a result of Ballico \cite[Theorem]{Ba}
(actually, the assertion that $H^3=3$ would imply the obvious case is not proved there).
However, assuming in our setting that either $Y$ is Fano or $X=\mathbb P^2$, we will see that the obvious
case is the only possibility (cf. Proposition \ref{Fano} and Theorem \ref{conclusion}).
\end{rmk}

More generally, with an eye to the characterization of projective manifolds admitting a given variety as a hyperplane section, Proposition \ref{very ample} suggests the following.
\begin{proposition} \label{hyp sec}
Let $\mathcal X \subset \mathbb P^N$ be a regular projective $n$-fold, with $n \geq 3$. If a general
surface section $\mathcal Y$ of $\mathcal X$ is a triple plane via the hyperplane bundle map, then either
\begin{enumerate}
\item $\mathcal X \subset \mathbb P^{n+1}$ is a smooth cubic hypersurface, or
\item $n=3$ and $\mathcal X \subset \mathbb P^5$ is the Segre product $\mathbb P^2 \times \mathbb P^1$.
\end{enumerate}
\end{proposition}
\begin{proof} Let $Z$ be a general $3$-dimensional linear section of $\mathcal X$ and set $\mathcal H=\mathcal O_{\mathbb P^N}(1)|_Z$, so that $\mathcal Y \in |\mathcal H|$.  Clearly, $h^0(Z,\mathcal H) =1 +h^0(\mathcal Y,\mathcal H_{\mathcal Y})$ since $h^1(\mathcal  O_Z)=h^1(\mathcal O_{\mathcal X})=0$, by the Lefschetz theorem. Hence
\begin{equation}
\Delta(Z, \mathcal H)=3+\mathcal H^3-h^0(Z, \mathcal H) = 3+\mathcal H_{\mathcal Y}^2 -1 -h^0(\mathcal Y, \mathcal H_{\mathcal Y}) \leq 1,  \nonumber
\end{equation}
since $\mathcal H_{\mathcal Y}^2=3$ and $h^0(\mathcal Y, \mathcal H_{\mathcal Y}) \geq 4$, $\mathcal H_{\mathcal Y}$ being a very ample divisor on $\mathcal  Y$. If $\Delta(Z,\mathcal H)$=1, then $Z \subset \mathbb P^4$ is a smooth cubic threefold by \cite [Corollary 6.7]{libroFujita} and then $\mathcal X$ is as in (1). On the other hand, if $\Delta (Z,\mathcal H)=0$, then $Z$ is the Segre product $\mathbb P^2 \times \mathbb P^1 \subset \mathbb P^5$, which, however, cannot ascend to higher dimensions. Actually, $\mathcal X$ is a scroll over $\mathbb P^1$ and then \eqref{grado} shows that $n=3$, i.e, $\mathcal X = Z$, as in (2).
\end{proof}

\medskip

From now on we will assume that our triple solid $Y$ has the additional structure of a scroll
over a smooth surface. So we will always refer to the setting \eqref{setting}.

\setcounter{theorem}{0}
\setcounter{equation}{2}
The structure of triple solid given by $\phi$, combined with the Chern--Wu relation implies:
\begin{equation}\label{Hcube}
3 = H^3 = c_1(\mathcal E)^2 - c_2(\mathcal E)
\end{equation}
$c_i(\mathcal E)$ denoting the $i$-th Chern class of $\mathcal E$.
So we have

\setcounter{theorem}{3}
\setcounter{equation}{0}

\begin{rmk}\label{rem0}
$\mathcal E$ is Bogomolov stable unless $(Y,H)$ is as in the obvious case.
Actually, \eqref{Hcube} says that $c_1(\mathcal E)^2-4c_2(\mathcal E)=3\big(1-c_2(\mathcal E)\big) \leq 0$,
because $c_2(\mathcal E)>0$ due to the ampleness of $\mathcal E$ \cite{BG};
therefore $\mathcal E$
is Bogomolov semistable. Moreover it is properly semistable if and only if $c_2(\mathcal E)=1$
and this occurs only for $(X,\mathcal E)=\big(\mathbb P^2, \mathcal O_{\mathbb P^2}(1)^{\oplus 2}\big)$ by \cite{LS}.
Hence, apart from the obvious case, $\mathcal E$ is Bogomolov stable.
\end{rmk}

Here we collect some properties of $Y$.

\begin{proposition} \label{prelim}
We have:
\begin{enumerate}
\item[(a)] $h^1(\mathcal O_Y)=0$;
\item[(b)] $X$ is a regular surface;
\item[(c)] a general element $S$ in the linear subsystem $\phi^*|\mathcal O_{\mathbb P^3}(1)| \subseteq |H|$
is a smooth regular surface;
\item[(d)] the ramification divisor $R$ of $\phi$ is very ample;
\item[(e)] $(Y,R)$ is a conic fibration over $X$ via $\pi$, with empty discriminant locus. In particular,
letting $P:=\mathbb P_X(\mathcal F)$, where $\mathcal F=\pi_*R$ and denoting by $\xi$ the
tautological line bundle and by $\widetilde{\pi}:P \to X$ the bundle projection, $Y$ is contained
in $P$ as a smooth divisor of relative degree $2$, belonging to the linear system
$|2\xi - 2 \widetilde{\pi}^*(K_X+ 2 \det \mathcal E)|$ and $\xi_Y=R$.
\end{enumerate}
\end{proposition}

\begin{proof}
(a) follows from \cite[Theorem 1]{Laz1}, and then equation \eqref{irreg} implies (b). As $H$ is ample
and $\phi^*|\mathcal O_{\mathbb P^3}(1)|$ is base-point free, its general element $S$ is a smooth
surface by the Bertini theorem: the fact that $h^1(\mathcal O_S)=0$ follows from the
Lefschetz theorem \cite{So}. This proves (c). The ramification formula says that
\begin{equation}
K_Y= \phi^*K_{\mathbb P^3} + R = -4H + R,  \nonumber
\end{equation}
hence $R=K_Y+4H$. Since $H$ is ample and spanned with $H^3=3$ it thus follows from \cite[Theorem 3.1]{LPS}
that $R$ is a very ample divisor. This gives (d). Finally, by the canonical bundle formula, we have
\begin{equation}
K_Y = -2H + \pi^*(K_X+\det \mathcal E),  \nonumber
\end{equation}
and by comparing the two expressions of $K_Y$ we get the relation
\begin{equation}
R = 2H + \pi^*(K_X+\det \mathcal E).   \nonumber
\end{equation}
The first assertion in (e) follows from the fact that $R$ restricts to every fiber of $\pi$ as
$\mathcal O_{\mathbb P^1}(2)$: the discriminant is empty since every fiber is irreducible.
Furthermore, as a conic fibration over $X$, $Y$ is contained as a smooth divisor of relative
degree $2$ inside $P:=\mathbb P_X(\mathcal F)$, where $\mathcal F=\pi_*R$;
more precisely, letting $\xi$ denote the tautological line bundle and $\widetilde{\pi}:P \to X$
the bundle projection extending $\pi$, we have that $Y \in |2\xi + \widetilde{\pi}^*\mathcal B|$
for some line bundle $\mathcal B$ on $X$ and $\xi_Y=R$.
Recalling that $\pi_*H = \mathcal E$, from the expression of $R$ we get
\begin{equation}
\mathcal F = \pi_*\big(2H+\pi^*(K_X+\det \mathcal E)\big) = S^2\mathcal E \otimes (K_X+\det \mathcal E),  \nonumber
\end{equation}
where $S^2$ stands for the second symmetric power. Since $\text{\rm{rk}}(\mathcal F)=3$, this gives
\begin{equation}
c_1(\mathcal F)=3  c_1(\mathcal E)+3(K_X+\det \mathcal E) = 3(K_X + 2 \det \mathcal E).  \nonumber
\end{equation}
The condition expressing the fact that the discriminant locus of $(Y,R)$ is empty is given by
$2c_1(\mathcal F)+3\mathcal B = \mathcal O_X$ \cite[p.\ 76]{BOSS}.
Therefore we get $\mathcal B=- \frac{2}{3}c_1(\mathcal F) =-2(K_X+2 \det \mathcal E)$, and this concludes the proof.
\end{proof}

\setcounter{equation}{0}

\begin{proposition} \label{A}
Suppose that $(Y,H)$ is not as in the obvious case.
Then $K_X+\det \mathcal E$ is ample and spanned.
\end{proposition}

\begin{proof}
Suppose that $K_X + \det\mathcal E$ is not ample. Then, according to \cite[Main Theorem]{Fu2},
$(X,\mathcal E)$ is one of the following pairs:
\begin{enumerate}
\item[(a)] $X$ is a $\mathbb P^1$-bundle over a smooth curve $C$ and
$\mathcal E_f=\mathcal O_{\mathbb P^1}(1)^{\oplus 2}$, for every fiber of the bundle projection $p:X \to C$;
\item[(b)] $\big(\mathbb P^2, \mathcal O_{\mathbb P^2}(2)\oplus \mathcal O_{\mathbb P^2}(1)\big)$;
\item[(c)] $(\mathbb P^2, T_{\mathbb P^2})$ (tangent bundle);
\item[(d)] $\big(\mathbb Q^2, \mathcal O_{\mathbb Q^2}(1)^{\oplus 2}\big)$.
\end{enumerate}
Note that the right hand term in equality \eqref{Hcube} is equal to $7$ in case (b) and $6$ in cases (c) and (d),
a contradiction. In case (a) we can set $X=\mathbb P_C(\mathcal V)$ where $\mathcal V$ is a rank-$2$
vector bundle over $C$ of degree $v:=\deg \mathcal V$, and up to a twist by a line bundle we can suppose
that $v=0$ or $-1$ according to whether it is even or odd, respectively; moreover, letting $\xi$ denote
the tautological line bundle and $p:X \to C$ the projection we have $\mathcal E = \xi \otimes \pi^* \mathcal G$
for some rank-2 vector bundle $\mathcal G$ on $C$. Set $\gamma:=\deg \mathcal G$. Then $\xi^2=v$,
$c_1(\mathcal E)^2 = (2 \xi + \gamma f)^2 = 4(v +\gamma)$, and
$c_2(\mathcal E) = \xi^2 + \gamma = v + \gamma$. Then equality \eqref{Hcube} gives $v + \gamma =1$.
But $c_2(\mathcal E)=1$ implies that $(X,\mathcal E)=\big(\mathbb P^2, \mathcal O_{\mathbb P^2}(1)^{\oplus 2}\big)$
by \cite{LS}. Thus $(Y,H)$ is as in the obvious case, a contradiction. Therefore
$K_X + \det\mathcal E$ is ample. Moreover,  it is also spanned in view of
\cite[Theorem A]{LM}, since $\mathcal E=\pi_*H$ is ample and spanned.
\end{proof}

\medskip
Recall that the triple cover $\phi: Y \to \mathbb P^3$ is said to be of triple section type if $Y$ is contained
in the total space of an ample line bundle on $\mathbb P^3$ as a triple section \cite{Fu0}.
As a consequence of Proposition \ref{A} we get the following conclusion (compare with \cite[Proposition 4.4]{LN3}).

\begin{corollary} \label{co}
$\phi$ is not of triple section type. In particular, $\phi$
is not a cyclic cover.
\end{corollary}

\begin{proof} If $\phi$ is of triple section type, then $K_Y = \phi^*\mathcal O_{\mathbb P^3}(k)=kH$
for some integer $k$, \cite[Proposition 3.2]{Laz1} (see also \cite[Theorem 2.1]{Fu0}).
Taking into account the canonical bundle formula we thus get $kH=-2H+\pi^*(K_X+ \det \mathcal E)$.
Therefore $k=-2$ and $K_X + \det \mathcal E=\mathcal O_X$, due to the injectivity of the homomorphism
$\pi^*: \text{\rm{Pic}}(X) \to \text{\rm{Pic}}(Y)$.
This conclusion, however, contradicts Proposition \ref{A}. Note also that it is not satisfied
even when $(Y,H)$ is as in the obvious case.
\end{proof}

\setcounter{equation}{0}

\begin{proposition} \label{E decomp}
If $\mathcal E$ fits into an exact sequence
\begin{equation}
0 \to M \to \mathcal E \to N \to 0, \nonumber
\end{equation}
where $M$ and $N$ are ample line bundles, then $(Y,H)$ can only be as in the obvious case.
In particular, except for that case, $\mathcal E$ is indecomposable.
\end{proposition}
\begin{proof} Assuming that $\mathcal E$ fits into an exact sequence as above, we have that
$c_1(\mathcal E)=M+N$ and $c_2(\mathcal E)=M \cdot N$. Thus \eqref{Hcube} becomes
$$3 = H^3 = c_1(\mathcal E)^2-c_2(\mathcal E) = M^2 + M \cdot N + N^2,$$
and  $M^2=M \cdot N=N^2=1$, because both $M$ and $N$ are ample. But then $(M-N) \cdot M=0$ and
$(M-N)^2=0$, hence the Hodge index theorem implies that $M$ and $N$ are numerically equivalent.
As $\mathcal E$ is spanned, $N$ is spanned too and then $(X,N)$ is a surface polarized by an ample
and spanned line bundle with $N^2=1$. Therefore $X=\mathbb P^2$ and
$M=N=\mathcal O_{\mathbb P^2}(1)$; then $\mathcal E=M \oplus N$ since $\text{Ext}^1(N,M)=
H^1(\mathbb P^2, \mathcal O_{\mathbb P^2})=0$.
\end{proof}

\setcounter{equation}{0}

Here is a consequence of Proposition \ref{E decomp}.

\begin{corollary} \label{c_2(E)}
If $(Y,H)$ is not as in the obvious case, then $c_2(\mathcal E) \geq 3$.
\end{corollary}

\begin{proof} Since $\mathcal E$ is ample and spanned we know that $c_2(\mathcal E) \geq 1$ with equality occurring
only in the obvious case, as already said.
Let $c_2(\mathcal E) =2$. Then a result of Noma \cite[Theorem 6.1]{No} shows that
either $\mathcal E$  is decomposable, which is impossible by Proposition \ref{E decomp}, or $X$ is not
a regular surface, which contradicts Proposition \ref{prelim} (b).
\end{proof}

Finally, we can prove

\begin{proposition} \label{Fano}
Suppose that $Y$ is a Fano threefold; then
$(Y,H)$ is as in the obvious case.
\end{proposition}

\begin{proof} Due to the assumption, $\mathcal E$ is a Fano bundle on $X$ \cite{SW}.
Let $\mathcal F$ be another rank-2 vector bundle on $X$ such that $Y=\mathbb P_X(\mathcal F)$.
Denoting by $\xi$ its tautological line bundle, from the
fact that $\mathcal E = \mathcal F \otimes \mathcal O_X(D)$ for some divisor $D$ on $X$,
we see that $H=\xi + \pi^*D$. Since $c_1(\mathcal E)= c_1(\mathcal F)+2D$ and
$c_2(\mathcal E) = c_2(\mathcal F) + c_1(\mathcal F)\cdot D + D^2$, we get from \eqref{Hcube} that
\begin{equation}\label{xicube}
3 = H^3 = c_1(\mathcal F)^2-c_2(\mathcal F) + 3D \cdot \big(c_1(\mathcal F)+D\big),
\end{equation}
hence $\xi^3 = c_1(\mathcal F)^2-c_2(\mathcal F)$ is also divisible by $3$. Moreover, the
fact that $Y$ is Fano implies that $X$ is a del Pezzo surface \cite[Proposition 1.5]{SW}.
We can therefore assume that $\mathcal F$ is normalized in an appropriate way.
Suppose that $(Y,H)$ is not as in the obvious case. Then, checking the list of the rank-$2$ Fano
bundles on surfaces \cite[Theorem]{SW} and taking into account  Proposition \ref{E decomp}
and \eqref{xicube} we see that if our $(X,\mathcal F)$ is in that list, then the possibilities for
$(X, \mathcal E)$, if any, restrict to the following:

\begin{enumerate}
\item[(1)] $X=\mathbb P^2$ and $\mathcal E$
is a stable spanned bundle fitting in an exact sequence
\newline
$0 \to \mathcal O_{\mathbb P^2}(-2)  \to \mathcal O_{\mathbb P^2}^{\oplus 3} \to \mathcal E \to 0$
 \ (case 7 in \cite[Theorem]{SW}; here $\mathcal E =\mathcal F(1)$);
\item[(2)] $X = \mathbb P^1 \times \mathbb P^1$ and $\mathcal E$
is a stable spanned bundle fitting in an exact sequence
\newline
$0 \to \mathcal O_{\mathbb P^1 \times \mathbb P^1}(-1,-1)  \to \mathcal O_{\mathbb P^1 \times \mathbb P^1}^{\oplus 3} \to \mathcal E \to 0$
 \ (case 12 in \cite[Theorem]{SW}; here $\mathcal E =\mathcal F(1,1)$).
\end{enumerate}

\noindent
However, in these cases the vector bundle $\mathcal E$ is not ample.
To see this suppose we are in case (1), consider the inclusion of
$Y=\mathbb P_{\mathbb P^2}(\mathcal E)$ in $\mathbb P^2 \times \mathbb P^2=
\mathbb P(\mathcal O_{\mathbb P^2}^{\oplus 3 })$
corresponding to the surjection $\mathcal O_{\mathbb P^2}^{\oplus 3} \to \mathcal E$ and
call $\rho:Y \to \mathbb P^2$ the restriction of the second projection $p_2$
of $\mathbb P^2 \times \mathbb P^2$ to $Y$ (note that $\pi$ is the restriction of the first projection).
Then, for the tautological line bundle of $\mathcal E$ we have that
$H=\rho^* \mathcal O_{\mathbb P^2}(1)$. Fix a point $x \in \mathbb P^2$:
then $\gamma:= \rho^{-1}(x)= p_2^{-1}(x) \cap Y$ is a curve inside $Y$ and clearly
if $\ell \subset \mathbb P^2$ is a general line, we get
$H \cap \gamma=\rho^{-1}(\ell) \cap \rho^{-1}(x) = \emptyset$.
Therefore $H$, hence $\mathcal E$, is not ample. The same argument applies to case (2) and this concludes the proof.
\end{proof}

%%%%%%%%%%%%%%%%%%%%%%%%%%%%%%%%%%%%%%%%%%%%%%%%%%%%%%%%%%%%%%%%%%%%%%%%%%%%%%					%%%%%%%%%%%%%%%%%%%%%
%%%%%%%%%%%%%%%%%%%                SECTION 5              %%%%%%%%%%%%%%%%%%%%%
%%%%%%%%%%%%%%%%%%%%%%%%%%%%%%%%%%%%%%%%%%%%%%%%%%%%%%%%%%%%%%%%%%%%%%%%%%%%%%

\section{Further constraints on $Y$ deriving from $S$ as triple plane} \label{constraints}
\setcounter{equation}{0}

\setcounter{equation}{0}

Let $Y$, $H$ and $\mathcal E$ be as in \eqref{setting}, and
let $S$ be a general element of the linear subsystem $\phi^*|\mathcal O_{\mathbb P^3}(1)| \subseteq |H|$
(recall that equality holds except when $(Y,H)$ is as in the obvious case).
Then $S$ is a smooth regular surface, by Proposition
\ref{prelim} (c), and the polarized surface $(S,H_S)$ inherits from $(Y,H)$ the structure of a triple plane
$\varphi:=\phi|_S:S \to \mathbb P^2$,
where $H_S = \varphi^*\mathcal O_{\mathbb P^2}(1)$. Moreover, referring to the scroll structure of
$(Y,H)$, by restricting the projection
$\pi:Y \to X$ to $S$ we get a birational morphism $r=\pi|_S: S \to X$. More precisely,
the pair $(S,H_S)$ has $(X, \det \mathcal E)$
as its adjunction theoretic minimal reduction, the reduction morphism being $r$.
This means that $S$ is a meromorphic non-holomorphic section of $\pi$
which contains $s > 0$ fibres $e_1, \dots ,e_s$
of $\pi: Y \to X$, and these curves, which are lines of $(S,H_S)$, are contracted by the birational morphism $r$ to a finite subset of $X$; in addition,
$X$ can contain no line with respect to $\det \mathcal E$, $\mathcal E$ being ample of rank $2$. Hence
$(X, \det \mathcal E)$ is the minimal reduction of $(S,H_S)$. In particular, $S$ is not minimal; moreover,
$H_S = r^* \det \mathcal E - \sum_{i=1}^s e_i$, so that
$(\det \mathcal E)^2 = c_1(\mathcal E)^2 = 3+s$, which combined with \eqref{Hcube}
shows that
\begin{equation} \label{s=c2}
s=c_2(\mathcal E).
\end{equation}
We have also $K_S = r^*K_X + \sum_{i=1}^s e_i$, hence $K_S+H_S = r^*(K_X+\det\mathcal E)$, which
has the following consequence on the sectional genus:
\begin{equation} \label{genus}
g(Y,H) = g(S,H_S)= g(X,\det \mathcal E).
\end{equation}
We set $g:=g(Y,H)$. Furthermore, $K_S^2=K_X^2-s$ and $e(S)=e(X)+s$.
Consider the exact sequence
\begin{equation} \label{exseq}
0 \to \mathcal O_Y \to H \to H_S \to 0. 
\end{equation}

By pushing \eqref{exseq} down via $\pi$ we get the sequence
\begin{equation}
0 \to \mathcal O_X \to \mathcal E \to \det \mathcal E \otimes \mathcal J_Z \to 0, \nonumber
\end{equation}
defined by the multiplication by $\theta$, the section of $\mathcal E$ that corresponds to the section of $H$ defining
$S$ in the isomorphism $H^0(Y,H) \cong H^0(X,\mathcal E)$. Here $Z$ stands for the zero locus of $\theta$ and
$\mathcal J_Z$ for its ideal sheaf. Recall that $Z$ consists of $s=c_2(\mathcal E)$ points of $X$, by \eqref{s=c2}.
Clearly,
\begin{equation}
h^0(\det \mathcal E \otimes \mathcal J_Z) = h^0(\det \mathcal E) - t  \nonumber
\end{equation}
where $t$ is the number of linearly independent linear conditions to be imposed
on an element of $|\det \mathcal E|$ to contain $Z$.
Of course, $t \leq \text{Card}(Z)=s$. On the other hand, recalling that $X$ is regular by Proposition
\ref{prelim} (b), we see from the cohomology of the exact sequence above that
\begin{equation}
h^0(\det \mathcal E \otimes \mathcal J_Z) = h^0(\mathcal E) -1 = h^0(H)-1=3,   \nonumber
\end{equation}
provided that $(Y,H)$ is not as in the obvious case. So we have
\begin{proposition}\label{triple plane}
Suppose that $(Y,H)$ is not as in the obvious case.
Then $\varphi:S \to \mathbb P^2$ factors through $r$ and the rational map
defined by the linear subsystem
of $|\det \mathcal E|$ of curves passing through $c_2(\mathcal E)$ points of $X$ that impose only
$h^0(\det \mathcal E)-3$ linearly independent linear conditions on them.
\end{proposition}

The following result will have relevant consequences.

\begin{proposition} \label{prop 4}
Suppose that $(Y,H)$ is not as in the obvious case.
Then $g \geq 3$, equality implying
$X=\mathbb P^2$
and $\mathcal E$ indecomposable of generic splitting type $(2,2)$, in particular semistable, with
$c_2(\mathcal E)=13$.
\end{proposition}

\begin{proof} Look at $(X, \det \mathcal E)$. Polarized
surfaces with sectional genus $\leq 1$ are well known \cite{libroFujita}. By
Proposition \ref{A} we know that $K_X + \det \mathcal E$ is ample and
spanned, since $(Y, H)$ is not as in the obvious case. We can
thus exclude that $(X, \det \mathcal E)$ is such a pair. Therefore $g \geq 2$.
However, it cannot be $g=2$, since  every ample and spanned rank 2 vector bundle of $c_1$-sectional genus $2$
(i.e., $g(X,\det \mathcal E)=2$) on a surface
is decomposable \cite[proof of the Theorem in the appendix]{FI}, but this is in contrast with Proposition \ref{E decomp}.
%Actually, according to \cite[Therem 1.3]{BLL} any ample and spanned rank 2 vector bundle of $c_1$-sectional genus $2$ on a surface is decomposable, but this is in contrast with Proposition \ref{E decomp}.
Finally, suppose that $g=3$. Then a close check of the list in \cite[Theorem 2.1, (III)]{FI}
combined with  Proposition \ref{E decomp} again
confines the possibilities to the following single case:  $X=\mathbb P^2$ and
$\mathcal E$ indecomposable with $\det \mathcal E =\mathcal O_{\mathbb P^2}(4)$.
%For the classification of ample and spanned vector bundles with $c_1$-sectional genus $3$ we can rely on
%\cite[Theorem 1.10]{BLL}. Recall that condition \eqref{Hcube} must hold in our setting. Thus a close check
%of the list confines the possibilities to the following case: $X=\mathbb P^2$ and $\det \mathcal E =\mathcal O_{\mathbb P^2}(4)$.
Now, let $(a_1,a_2)$, with $a_1 \geq a_2$, be the
generic splitting type of $\mathcal E$ (i.e., $\mathcal E_{\ell} = \mathcal O_{\ell}(a_1) \oplus \mathcal O_{\ell}(a_2)$
for the general line $\ell \subset \mathbb P^2$). Clearly $(a_1,a_2)=(3,1)$ or $(2,2)$, due to the ampleness.
Suppose that
$(a_1,a_2)=(3,1)$. Then $\mathcal E$ has no jumping lines \cite[p.\ 29]{OSS}, so that it is uniform.
Thus, according to a theorem of Van de Ven \cite[p.\ 211]{OSS}, $\mathcal E$ is either
$\mathcal O_{\mathbb P^2}(3) \oplus \mathcal O_{\mathbb P^2}(1)$,
or a twist of the tangent bundle. Both cases, however, have to be excluded: the former would contradict
Proposition \ref{E decomp}, while
in the latter $\det \mathcal E$ could not be $\mathcal O_{\mathbb P^2}(4)$. Therefore $(a_1,a_2)=(2,2)$.
It thus follows from \cite[Lemma 2.2.1, p.\ 209]{OSS} that $\mathcal E$ is semistable. Finally, \eqref{Hcube}
implies $c_2(\mathcal E)=13$.
\end{proof}
%\red{[Be careful. Relying on \cite{BLL}, we cannot look directly at $(Y,H)$ and use \cite[Theorem 2.6]{BLL}
% because the case $d=3$ we are interested in here is missed there, due to a wrong application of (0.4) in (2.6).]}
\setcounter{equation}{0}

\begin{rmk}
Note that spanned rank-2 vector bundles on $\mathbb P^2$ with Chern classes $(c_1,c_2)=(4,13)$
do exist according to \cite[Theorem 0.1]{El}.
Anyway, the general stable rank-$2$ vector bundle with these Chern classes is certainly not spanned,
since its invariants do not satisfy the conditions in \cite[Theorem 2.6]{HK}. As a consequence, \cite[Theorem 5.1]{HK}
is not applicable to establish the ampleness. Moreover, for a vector bundle like $\mathcal E$, giving rise to
a pair $(Y,H)$ with $g=3$, if any, we know that $h^0(\mathcal E)=4$ and by the Riemann--Roch theorem combined
with the exact cohomology sequence induced by \eqref{exseq} it follows easily that $h^1(\mathcal E)=1$.
Therefore, such an $\mathcal E$, if any, would be quite special in moduli by the Weak Brill--Noether theorem for $\mathbb P^2$
\cite[Theorem 2.4]{HK}. In fact, such a vector bundle does not exist, according to what
we will prove in Section \ref{overP2}.
\end{rmk}

Now let us focus on the triple plane
$\varphi:S \to \mathbb P^2$ induced by $\phi$, deriving further restrictions on $Y$.
Let $B$ be the branch locus and let $\mathcal T$ be the rank 2 vector bundle
on $\mathbb P^2$ such that $\varphi_*\mathcal O_S = \mathcal O_{\mathbb P^2}\oplus \mathcal T$,
i.e.\ the Tschirnhaus bundle of $\varphi$.
Set $b_i=c_i(\mathcal T)$. Then $B \in |2 \det \mathcal T^{\vee}|$ so that $b:= \deg B = -2b_1 >0 $;
moreover, if $\varphi$ is general in the sense of \cite[p.\ 1154]{Mi3ple}, then $B$ is irreducible
and has only cusps as singularities, their number being $c = 3b_2$ \cite[Lemma 10.1]{Mi3ple}.
%by general triple cover, Miranda means that $\phi$ has no total ramification in codimension one
%(which is obvious in our case since the image surface of $\varphi$ is $\mathbb P^2$; in general,
%by a result of S.L. Tan [Triple covers on smooth algebraic surfaces, in "Geometry and Nonlinear Partial
%Differential Equations (Hangzhou,2001)", AMS/IP Stud. Adv. Math. 29, Amer. Math. Soc., Providence, RI, 2002,
%Theorem 1.3 and Theorem 3.2], for a triple surface $f:S \to Z$, with $Z$ smooth, the smoothness of $S$ implies
%that $B = B'+B''$, where $B'$ is smooth, $B''=0$ if $f$ is totally ramified, $B'$ and $B''$ have no common
%points and $B''$ has only cusps as singular points where $f$ is totally ramified; so $Z=\mathbb P^2$
%implies that $B'=0$) and that the only singular points of the branch locus $B$ are ordinary cusps.
%Hence in our case $\varphi$ general simply means that $B$ has only ordinary cusps as singularities.
Furthermore, the Riemann--Hurwitz theorem applied to $\varphi^{-1}(\ell)$, where $\ell \subset \mathbb P^2$
is a general line, gives
\setcounter{theorem}{0}
\setcounter{equation}{3}

\begin{equation} \label{b}
b = 2g+4.
\end{equation}

\setcounter{theorem}{3}
\setcounter{equation}{0}

As an immediate consequence of Proposition \ref{prop 4} we have

\begin{rmk}\label{b>=10}
If $(Y,H)$ is not as in the obvious case, then $b \geq 10$, equality implying $g=3$.
\end{rmk}

Consider the equalities $h^i(\mathcal O_S) = h^i(\mathcal O_{\mathbb P^2}) + h^i(\mathcal T)$
coming from the definition of $\mathcal T$.
For $i=1$, since $S$ is regular we get $h^1(\mathcal T)=0$.
On the other hand we know that $h^0(\mathcal T)=0$, hence letting $i=2$ we see that
$h^2(\mathcal O_S) = h^2(\mathcal T)= \chi(\mathcal T)$, which can be computed with the Riemann--Roch theorem
\cite[p.\ 26]{BHPV}. In conclusion, we obtain
$p_g(S) = \frac{1}{8} b(b-6) + 2 - \frac{c}{3}$, hence
\begin{equation}
\chi(\mathcal O_S) = 1 + p_g(S) = \frac{1}{8} b(b-6) + 3 - \frac{c}{3}.  \nonumber
\end{equation}
On the other hand, Miranda's formulas \eqref{M form}, rewritten for $S$ in terms of $b$ and $c$,
provide the following values of $K_S^2$ and $e(S)$:
\setcounter{theorem}{0}
\setcounter{equation}{4}
\begin{equation} \label{M formulas}
K_S^2 = 27 - 6b + \frac{1}{2}b^2 - c \qquad \text{and} \qquad  e(S)= 9 - 3b + b^2 - 3c .
\end{equation}
\setcounter{theorem}{4}
\setcounter{equation}{0}
%Summing up these two relations and taking into account Noether's formula, we get
%$\chi(\mathcal O_S) = \frac{1}{12}\big(K_S^2+e(S)\big)= \frac{1}{8}b(b-6) + 3 - \frac{c}{3}$,
%which confirms the value computed before.] 
Recalling that $r:S \to X$ is a birational morphism
which factors through $s$ blowing-ups, this immediately gives the corresponding numerical characters of $X$.

\medskip
It is useful to recall that for $(Y,H)$ as in the obvious case, the pair $(S,H_S)$ is as in (2)
of Theorem \ref {maintheorem}. In particular, we have $g=0$; moreover, for
the triple plane $\varphi:S \to \mathbb P^2$ induced by $\phi$,
the Tschirnhaus bundle is $\mathcal T = \mathcal O_{\mathbb P^2}(-1)^{\oplus 2}$ \cite[Table 10.5]{Mi3ple}.
Then the branch curve $B$ of $\varphi$ is a quartic, since $b=-2b_1=4$,
and \eqref{M formulas} shows that $c=3$. Furthermore,
$\varphi$ maps the only $(-1)$-line of $(S, H_S)$
(namely the only fiber of $\pi:Y \to X$ that $S$ contains), isomorphically to a line $\ell \subset \mathbb P^2$,
which is bitangent to $B$ (there is only one bitangent line in this case, by Pl\"ucker formulas).

\medskip
Coming back to the general case, a natural question
concerning the Tschirnhaus bundle of $\varphi$ is what happens
when $\mathcal T$ is decomposable, namely
$\mathcal T = \mathcal O_{\mathbb P^2}(-m) \oplus \mathcal O_{\mathbb P^2}(-n)$
for some positive integers $m$, $n$, as in \cite[p.\ 1156]{Mi3ple}.
As we have seen, $(m,n)=(1,1)$ corresponds to $(Y,H)$ being as in the obvious case.
We have $b=2(m+n)$, $c=3mn$ and we can rewrite the invariants of $S$ in terms of $m,n$ as in \cite[Corollary 10.4]{Mi3ple}.
In particular, we have
$p_g(S)= (\frac{1}{2})(m^2+n^2 -3m -3n)+2$,
$K_S^2=2(m+n-3)^2-3(mn-3)$, $e(S)=4(m+n)^2 - 6(m+n)-9(mn-1)$.
Then we immediately obtain the following result.

\setcounter{equation}{0}

\begin{proposition} \label{T decomp}
If $\mathcal T$ is decomposable and $X$ is a surface with $p_g(X)=0$,
then $(Y,H)$ is necessarily as in the obvious case.
\end{proposition}

\begin{proof} Since $p_g$ is a birational invariant, we have $p_g(S)=0$.
Hence $(m,n)$ must be an integral point
of the curve $\Gamma$ represented in the $(m,n)$-plane by the equation
\begin{equation}
m^2+n^2-3m-3n+4=0.  \nonumber
\end{equation}
Note that $\Gamma$ is a circle centered at $(\frac{3}{2},\frac{3}{2})$ with radius $\frac{1}{\sqrt{2}}$; hence its
integral points are $(1,1)$, $(1,2)$, $(2,1)$, $(2,2)$ only. In view of the symmetry between $m$ and $n$ we
can confine to consider the three pairs $(m,n)=(1,1), (1,2), (2,2)$.
In all these cases we have $b = 2(m+n) \leq 8$, hence the assertion follows from Remark \ref{b>=10}.
\end{proof}

\setcounter{equation}{0}

Still about the branch curve $B$, we have

\begin{proposition} \label{range of c}
Let things be as in the setting \eqref{setting},
let $S$ be a smooth element of $\phi^*|\mathcal O_{\mathbb P^3}(1)|$,
and suppose that
$\varphi:S \to \mathbb P^2$ is a general triple plane, then
\begin{equation}
\frac{1}{6}\ b^2  <  c \leq \min \Big \lbrace  \frac{1}{16}\ b(5b-6) - \frac{s}{2}\  ,\  \frac{3}{8}\ b(b-6)+6  \Big \rbrace.  \nonumber
\end{equation}
\end{proposition}

\begin{proof} To prove the lower bound for $c$, note that the ramification divisor
of $\varphi$ is $R_S:=R \cap S$, $R$ being the ramification divisor of $\phi$. So, $\varphi(R_S)=B$.
As $H_S=\varphi^* \mathcal O_{\mathbb P^2}(1)$ we have $H_S \cdot R_S = \deg B = b$.
Hence the Hodge index theorem gives the inequality
$b^2 = (R_S \cdot H_S)^2 \geq H_S^2 R_S^2 = 3 R_S^2$. On the other hand $R_S=K_S+3H_S$ by the ramification formula.
Having all the ingredients, recalling \eqref{b} and the expression of $K_S^2$, we can thus compute
\begin{equation}
R_S^2 = (K_S+3H_S)^2 = K_S^2 + 6(2g-2)+3H_S^2 = \frac{1}{2}\ b^2 - c. \nonumber
\end{equation}
Then the above inequality says that
$c \geq \frac{1}{6}\ b^2$ (compare with \cite[Corollary 2.7]{FPV})
%\red{[$\blacksquare$ il piano triplo considerato dai tre autori prescinde da $(Y,H)$, perci\`o \`e naturale che
%per loro possa sussistere anche l'uguaglianza; cf.\ anche il caso 2 nella Table \ref{Table1}]}.
Now suppose that equality holds. Then $R_S$ and $H_S$ are linearly dependent in $\text{\rm{NS}}(S) \otimes \mathbb Q$.
But this implies that either $H_S \equiv tK_S$ for some rational $t$ or $K_S \equiv 0$. The latter case cannot
occur, since $S$ is not minimal. In the former case, for a $(-1)$-curve $e \subset S$
we have $0 < H_S e = t K_S e = -t$, hence $t$ is negative. Then
$-K_S \equiv \frac{p}{q}H_S$ for some positive $\frac{p}{q} \in \mathbb Q$.
This means that $S$ is a del Pezzo surface: in particular
we have that $-K_S = \frac{p}{q}H_S$. By combining the classification of these surfaces
with the fact that $S$ is not minimal we argue that $-K_S$ is not divisible in $\text{\rm{NS}}(S)$.
Note that the same is true for $H_S$, since $H_S^2=3$. Thus the above equality allows us to conclude that $-K_S=H_S$,
hence $g=1$. But this is impossible in view of Proposition \ref{prop 4}, taking also into account that
$g=0$ if $(Y,H)$ is as in the obvious case. Therefore the inequality we obtained above is strict.

As to the upper bounds for $c$, the one with respect to the second term in the $\min$ derives
from the obvious inequality $p_g(S) \geq 0$ combined with the expression of $p_g(S)$.
To prove the other bound we use the inequality $K_X^2 \leq 3e(X)$.  Recall that if $X$
is a surface of general type, this is just the Bogomolov--Miyaoka--Yau inequality \cite[p.\ 275]{BHPV},
while if $X$ has Kodaira dimension $\leq 1$ then the above inequality follows immediately from the theory
of minimal models, simply recalling that $X$ is regular, due to Proposition \ref{prelim} (b).
On the other hand, since $S$ is obtained from $X$ via $s$ blowing-ups we have that
$3e(S)-K_S^2 = 4s +\big(3e(X)-K_X^2\big) \geq 4s$,
Due to the expression of both $K_S^2$ and $e(S)$ provided by \eqref{M formulas}
we can immediately convert this inequality into the bound with respect to the first term in the $\min$.
\end{proof}

\setcounter{equation}{0}

\noindent {\it{Comments}}. i)  Concerning the upper bound with respect to the first term of
the $\min$ in Proposition \ref{range of c} one can say a bit more if $X$ is not of
general type, since instead of looking at $3e(S)-K_S^2$ one
can use better lower bounds for $2e(S)-K_S^2$ in terms of $s$, according to the Kodaira dimension.
In particular, if $S$ is rational, then the upper bound for $c$ in Proposition \ref{range of c} can be improved.
Actually, $S \not= \mathbb P^2$, hence, there exists a birational morphism $S \to S_0$, where
$S_0$ is a Segre--Hirzebruch surface (either a minimal model or $\mathbb F_1$).
% letting $S_0$ denote a minimal model we can suppose that $S_0$
%is the Segre--Hirzebruch surface $\mathbb F_{\varepsilon}$, for some $\varepsilon \geq 0$,
Then $K_S^2=8-t$, $e(S)=4+t$ for some $t \geq 0$ (the number of blow-ups factoring
this birational morphism). Thus $2e(S)-K_S^2=3t$. Moreover, $s \leq t$ since $X$
is not necessarily minimal, unless $X=\mathbb P^2$, in which case $s=t+1$. So, apart from this case,
$s \leq t = \frac{1}{3}\big(2e(S)-K_S^2\big)$ and
taking into account \eqref{M formulas}, this gives
$c \leq \frac{3}{10}b^2 - \frac{3}{5} (s+3)$.
%$c \leq \frac{3}{10}(b^2-12)$. It is immediate to check that this bound
%is sharper than that in Proposition \ref{range of c} for $b \leq 7$ and for $b \geq 23$.

\noindent
ii) According to \cite[Corollary 2.7]{FPV} the inequality $c \geq \frac{1}{6}b^2$ holds for any general triple plane.
We emphasize that the inequality proved in Proposition \ref{range of c} is strict because it  refers only to
triple planes deriving from a triple solid as in \eqref{setting}.

\medskip

We conclude this Section with a general property that the pair
$(X,\mathcal E)$ has to satisfy if $Y$ is as in our setting.
Recall that $\mathcal E$ is ample and spanned of rank $2$, and $h^0(\mathcal E) \geq 4$
(with equality except when $(Y,H)$ in the obvious case);
so let $V$ be a $4$-dimensional vector subspace of $H^0(X,\mathcal E)$ spanning $\mathcal E$
and let $\mathbb G:=\mathbb G(1,3)$ be the grassmannian of the codimension $2$
vector subspaces of $V$. According to \cite[Remark 2.6]{Ar}, since $\mathcal E$ is ample and spanned by $V$,
$\mathcal E$ defines a morphism $\psi:X \to \mathbb G$, finite to its image $W:=\psi(X)$,
such that $\mathcal E=\psi^*\mathcal Q$, where $\mathcal Q$ is the universal rank-2 quotient bundle of $\mathbb G$.

\begin{proposition} \label{G(1,3)}
Consider the morphism $\psi:X \to \mathbb G$ defined by $\mathcal E$, and write
$W= \alpha\Omega(0,3) + \beta\Omega(1,2)$,
as a linear combination of the usual Schubert cycle classes with integral coefficients
$\alpha =  W \cdot \Omega(0,3)$ and $\beta = W \cdot \Omega(1,2)$ in the cohomology ring of $\mathbb G$.
Then $s=c_2(\mathcal E)=\beta \deg \psi$ and $3=\alpha \deg \psi$. In particular,
if $s$ and $3$ are coprime, then $\psi$ is birational and $W$ has bidegree $(3,s)$.
Moreover, if $\psi$ is an embedding, then the following relation holds,
connecting $c_2(\mathcal E)$ with the Chern classes of the Tschirnhaus bundle of $\varphi$:
\begin{equation}
(s-2)(s-3) = -2b_1^2 - 2b_1 + 6b_2. \nonumber
\end{equation}
\end{proposition}

\begin{proof} Writing $W= \alpha\Omega(0,3) + \beta\Omega(1,2)$ as we said, and recalling that
$c_1(\mathcal Q)= \mathcal O_{\mathbb G}(1)$ and $c_2(\mathcal Q)=\Omega(1,2)$,
by the functoriality of the Chern classes we get
\begin{equation}
c_2(\mathcal E)=\psi^*c_2(\mathcal Q_W)=\deg \psi \ \Big(\Omega(1,2)
\cdot \big(\alpha \Omega(0,3) + \beta \Omega(1,2)\big) \Big)
=\beta \ \deg \psi.  \nonumber
\end{equation}
and
\begin{equation}
c_1(\mathcal E)^2= \psi^*\big(\mathcal O_{\mathbb G}(1)_W\big)^2  =  \deg \psi \ \deg(W) =  (\alpha+\beta) \deg \psi.  \nonumber
\end{equation}
Therefore, \eqref{Hcube} gives $3=\alpha \deg \psi$ and this, in turn, combined with \eqref{s=c2} proves
the assertion on the birationality of $\psi$ and the bidegree of $W$.
Finally, if $\psi$ is an embedding, the ``formule clef\ " applied to the smooth congruence
$X \cong W$ in $\mathbb G$
\cite[Proposition 2.1]{AS} implies
\begin{equation}
9+s^2 = 3(3+s) + 4(2g-2) + 2K_X^2 -12\chi(\mathcal O_X). \nonumber
\end{equation}
Taking into account \eqref{b} and the expressions of $K_X^2$ and $\chi(\mathcal O_X)$
deriving from \eqref{M formulas} in view of the birationality between $S$ and $X$,
this proves the final relation.
\end{proof}

\begin{rmk}
We emphasize that $\psi: X \to \mathbb G$ can be an embedding although $\mathcal E$ is not very ample (see
\cite[Proposition 2.4 and Remarks 2.5 and 2.6]{Ar}). However, if $(Y,H)$ is as in the obvious case, then $\mathcal E$ is very ample,
$\psi$ is in fact an embedding, and $W$ is the Veronese surface; in this case both sides of the equality
in the last display are equal to $10$.
\end{rmk}

%%%%%%%%%%%%%%%%%%%%%%%%%%%%%%%%%%%%%%%%%%%%%%%%%%%%%%%%%%%%%%%%%%%%%%%%%%%%%%
%%%%%%%%%%%%%%%%%%%                SECTION 6             %%%%%%%%%%%%%%%%%%%%%
%%%%%%%%%%%%%%%%%%%%%%%%%%%%%%%%%%%%%%%%%%%%%%%%%%%%%%%%%%%%%%%%%%%%%%%%%%%%%%

\section{Scrolls over $\mathbb P^2$} \label{overP2}

\setcounter{equation}{0}

Let $(Y,H)$, $\pi:Y \to X$, and $\mathcal E$ be as in our setting \eqref{setting} again and let
$\varphi:S \to \mathbb P^2$ be the triple plane induced by $\phi$ as in Section \ref{constraints}.
When $X=\mathbb P^2$, the possibilities for $\mathcal E$, $b$ and $c$ are extremely restricted.

\begin{proposition} \label{P2}
Let things be as above and suppose that
$\varphi:S \to \mathbb P^2$ is a general triple plane. If $X=\mathbb P^2$,
%Let $\phi:Y \to \mathbb P^3$ be a triple \red{solid}, let $H=\phi^*\mathcal O_{\mathbb P^3}(1)$,
%and suppose that $\varphi:S \to \mathbb P^2$ is a general triple plane.
%If $(Y,H)$ is a scroll over $\mathbb P^2$ and $\mathcal E=\pi_*H$, where $\pi:Y \to \mathbb P^2$
%is the scroll projection,
then either $(Y,H)$ is as in the obvious case or it has the following characters:
\begin{equation} \label{candidate}
c_1(\mathcal  E)=\mathcal O_{\mathbb P^2}(4), \quad s=13, \quad b=10, \quad c=21.
\end{equation}
\end{proposition}

\begin{proof} For $X=\mathbb P^2$ we have $3e(S)-K_S^2= 3e(X)-K_X^2+4s=4s$.
Then \eqref{M formulas} combined with the expression $c=\frac{3}{8}\ b(b-6)+6$ deriving from the
condition $p_g(S)=0$ leads to the equation $b^2-30b+8(12+s)=0$. Hence $b = 15 \pm \sqrt{D}$,
where $D=129-8s$. Imposing that $D$ is non-negative we get the bound $s \leq 16$ and then the list of the admissible
values of $s$ follows by requiring that $D$ is the square of an integer. These values, together with
the corresponding $b$ and $c$ deriving from the above relations, are summarized in Table 1 below.
% and the remaining invariants in the table can be recovered by \eqref{M formulas}.
%The assertion concerning $e(S)$ and $K_S^2$ is obvious since $S$ is $\mathbb P^2$ blown-up at $s$ points.

{\small
\begin{table}[!h]
\begin{center}
\begin{tabular}{ | c | c | c | c | c | c | c | c | c | c | c | c | c | }    \hline
\text{\rm{case}} & $1$ & $2$ & $3$ & $4$ &  $5$ & $6$ & $7$ & $8$ & $9$ & $10$ & $11$ & $12$   \\  \hline  \hline
                   $s$   & $1$ & $6$ & $10$ & $13$ & $15$ & $16$ & $16$ & $15$ & $13$ & $10$ & $6$ & $1$ \\  \hline
                 $b$   & $4$ & $6$ & $8$ & $10$ & $12$ & $14$ & $16$ & $18$ & $20$ & $22$ & $24$ & $26$ \\  \hline
                  $c$   &  $3$ & $6$ & $12$ & $21$ & $33$ & $48$ & $66$ & $87$ & $111$ & $138$ & $168$ & $201$ \\  \hline
\end{tabular}
\smallskip
\caption{
%Numerical invariants for $X=\mathbb  P^2$.
}
\label{Table1}
\end{center}
\end{table} }

\noindent
Clearly case 1 corresponds to $(Y,H)$ being as in the obvious case while case 4 corresponds
to the further possibility mentioned in the statement.
So, it is enough to show that all remaining cases cannot occur. Clearly cases 2 and 3 are ruled out by Remark
\ref{b>=10}. Consider the remaining cases 5--12 and set $c_1(\mathcal E) = \mathcal O_{\mathbb P^2}(a)$.
Recalling \eqref{b}, \eqref{genus} and  the fact that $\big(\mathbb P^2, \mathcal O_{\mathbb P^2}(a)\big)$ is the minimal
reduction of $(S,H_S)$, Clebsch formula implies that  $g=\frac{1}{2}b-2 = \frac{1}{2} (a-1)(a-2)$. This rules out
all cases except cases 7 and 11, in which we get $a=5$ and $6$ respectively.
However, computing $c_1(\mathcal E)^2-c_2(\mathcal E)=a^2-s$ in these cases we see that
condition \eqref{Hcube} is not satisfied.
\end{proof}

Now suppose that $(Y,H)$ is a scroll over $X=\mathbb P^2$ with the characters as in \eqref{candidate}.
For the description of the triple plane $\varphi:S \to \mathbb P^2$ in this case we refer to
\cite[2.2 and 3.4]{FPV}. We have $\mathcal T = T_{\mathbb P^2}(-4) = \Omega^1_{\mathbb P^2}(-1)$,
in view of the natural identification $\Omega^1_{\mathbb P^2}\cong T_{\mathbb P^2}\otimes \det \Omega^1_{\mathbb P^2}
=T_{\mathbb P^2}(-3)$. In particular, $\mathcal T$ is stable.
Moreover, we can observe that in this case the triple plane $\varphi:S \to \mathbb P^2$
is general, regardless of the assumption made in Proposition \ref{P2}. Actually,
the vector bundle $S^3\mathcal T^{\vee} \otimes \det \mathcal T=
S^3\big(T_{\mathbb P^2}(-1)\big) \otimes \mathcal O_{\mathbb P^2}(1)$ is spanned, due to the Euler sequence. Thus
by combining \cite[Theorem 1.1]{Mi3ple} and \cite[Theorem 2.1 and Theorem 3.2]{T} with the fact
that Corollary \ref{co} prevents $\varphi$ from being totally ramified, we conclude that $\varphi:S \to \mathbb P^2$ is general.
By \eqref{b} we get $g=3$, since $b=10$, and then Proposition \ref{prop 4} applies.
We know that $(X,\det \mathcal E)=\big(\mathbb P^2, \mathcal O_{\mathbb P^2}(4)\big)$, hence
$h^0(\det \mathcal E)=15$ and therefore Proposition \ref{triple plane} tells us that
the triple plane $\varphi:S \to \mathbb P^2$ is defined via
the linear system of plane quartics passing
through $13$ points that impose only $12$ independent
linear conditions on them (see also \cite[Proposition 3.7]{FPV} and \cite[p.\ 1158]{Mi3ple}).
Clearly, such a triple plane exists.
%Moreover, if it derives from a triple solid $(Y,H)$ as in our setting, then Proposition \ref{G(1,3)} tells us
%that there exists a birational morphism $\psi$ from $\mathbb P^2$ to a congruence
%$W \subset \mathbb G$ of bidegree $(3,13)$ and $\mathcal E= \psi^*\mathcal Q_W$, where $\mathcal Q$
%is the universal rank-$2$ quotient bundle of $\mathbb G$.
%Recall that the universal rank-2 quotient bundle $\mathcal Q$ is (nef but) not ample on $\mathbb G$;
%thus the above description does not guarantee the ampleness of $\mathcal E$
%hence the effectiveness of a pair $(Y,H)$ with the characters as in \eqref{candidate}.
However, it cannot derive from a pair $(Y,H)$ as in our setting.
To see this, let $\widetilde{\mathcal T}$ be the Tschirnhaus bundle of $\phi$,
i.e. $\phi_* \mathcal O_Y = \mathcal O_{\mathbb P^3} \oplus \widetilde{\mathcal T}$.
If $\Pi=\mathbb P^2 \subset \mathbb P^3$ is the
plane such that $S=\phi^{-1}(\Pi)$, then
\begin{equation}
\mathcal O_{\mathbb P^2} \oplus \mathcal T = \varphi_*\mathcal O_S = \phi_*(\mathcal O_Y|_S) =
(\mathcal O_{\mathbb P^3} \oplus \widetilde{\mathcal T})|_{\Pi} = \mathcal O_{\Pi} \oplus
\widetilde{\mathcal T}|_{\Pi},  \nonumber
\end{equation}
hence $\widetilde{\mathcal T}|_{\mathbb P^2}=\mathcal T$,
and therefore $c_i(\mathcal  T)= c_i(\widetilde{\mathcal T})|_{\mathbb P^2}$.
We know that
%that in our case $\mathcal T = T_{\mathbb P^2}(-4)$, which is indecomposable. As a consequence, $\widetilde{\mathcal T}$ has to be an indecomposable rank-2 vector bundle on $\mathbb P^3$, and since
$b_1=-\frac{b}{2}=-5, b_2=\frac{1}{3}c=7$ by Proposition \ref{P2}. As a consequence, $\widetilde{\mathcal T}$
has Chern classes $-5h$ and $7h^2$,
respectively, where $h=\mathcal O_{\mathbb P^3}(1)$.
But this contradicts the Schwarzenberger condition $c_1 \cdot c_2 \equiv 0$ (mod. 2),
necessary for the existence of a rank 2 vector bundle on $\mathbb P^3$
\cite[p.\ 113]{OSS}.

\medskip

In conclusion, we have the following result.

\begin{theorem}\label{conclusion}
Let $\phi:Y \to \mathbb P^3$ be a triple solid, $H=\phi^*\mathcal O_{\mathbb P^3}(1)$ and suppose
that $\varphi:S \to \mathbb P^2$ is a general triple plane.
If $Y$ is a scroll over $\mathbb P^2$ for some polarization, then $(Y,H)$ is
necessarily as in the obvious case.
\end{theorem}

The above result does not mean that the pair $(Y,H)$ as in the obvious case
is the only scroll over $\mathbb P^2$ containing a smooth surface which is a triple plane.
From this perspective we would like to emphasize the following fact.
\begin{rmk}\label{final}
Given a scroll $(Y,L)$ over $\mathbb P^2$ for some ample line bundle $L$, 
which is not as in the obvious case, it may happen that $Y$ contains a smooth surface $S$ such that:
i) $S$ has the structure of a triple plane, ii) $M:=\mathcal O_Y(S)$ is a very ample line bundle, and
iii) $M\not= L$.
To give an example, consider $Y:=\mathbb P(T_{\mathbb P^2})$. Recalling that
$Y$ is contained in $P:= \mathbb P^2_1 \times \mathbb P^2_2$ as a smooth element of $|\mathcal O_P(1,1)|$,
we see that $Y$ has two distinct structures of $\mathbb P^1$-bundle over $\mathbb P^2$, $\pi_i:Y \to \mathbb P^2_i$
$(i=1,2)$, induced by the projections of $P$ onto the two factors. Set $L:=\big(\mathcal O_P(1,1)\big)_Y$; then
$L$ is very ample, and we can regard $(Y,L)$ as a scroll over $\mathbb P^2$, e.g.\ via $\pi_1$.
As is well-known, the general element $\Sigma \in |L|$ is a del Pezzo surface of degree $6$
and $\pi_1|_{\Sigma }:\Sigma \to \mathbb P^2_1$ is a birational morphism consisting of the blow-up a three
general points. Now look at $\text{\rm{Pic}}(Y)$, which can be generated by
$L$ and $h:=\pi_1^*\mathcal O_{\mathbb P^2_1}(1)$. The line bundle $M:=L+2h$
is clearly very ample. Let $S \in |M|$ be a general element: then $S$ is a smooth surface,
and $\varphi:=\pi_2|_S:S \to \mathbb P^2_2$ is a triple plane. Actually, $\varphi$ is a finite morphism
and recalling that $\mathcal O_P(1,0)^3=\mathcal O_P(0,1)^3 =0$
and  $\mathcal O_P(1,0)^2 \cdot \mathcal O_P(0,1)^2 =1$, we see that its degree is computed by
\begin{equation}
S \cdot \big(\mathcal O_P(0,1)_ Y\big)^2 =  (L+2h) \cdot  \big(\mathcal O_P(0,1)_ Y\big)^2 = \mathcal O_P(3,1) \cdot
\mathcal O_P(1,1) \cdot \big(\mathcal O_P(0,1)\big)^2 = 3. \nonumber
\end{equation}
\end{rmk}

Finally, restricting our attention to triple solids with sectional genus $3$, we want to stress that Theorem
\ref{conclusion} constitutes a significant progress
compared with \cite[Proposition 3.3]{LL}.

\medskip
{\bf{Acknowledgments.}} The first author is grateful to Prof.\ Jack Huizenga for useful correspondence
on the range of applicability of the ampleness criterion in \cite[Theorem 5.1]{HK} in connection with
a vector bundle corresponding to \eqref{candidate}. Both authors are members of INdAM (G.N.S.A.G.A.).

\providecommand{\bysame}{\leavevmode\hbox
to3em{\hrulefill}\thinspace}

\newpage


\begin{thebibliography}{10}

\bibitem  {Ar} { Enrique Arrondo},  {\it Subvarieties of Grassmannians},
Lect.\ Notes Series Dipart.\ di Matematica Univ.\ Trento {\bf 10}, 1996.

\bibitem {AS}  {Enrique Arrondo and Ignacio Sols}, {\it On congruences of lines in the projective space} (Chapter
6 written in collaboration with M. Pedreira),
M\'em.\ Soc.\ Math.\ France (2e s\'erie) {\bf 50} (1992).

\bibitem {Ba} {Edoardo Ballico}, {\it On ample and spanned vector bundles with zero $\Delta$-genera},
Manuscripta Math.\ {\bf{70}} (1991), 153--155.

\bibitem {BHPV} {Wolf P.\ Barth, Klaus Hulek, Chris\ A.M.\ Peters, and Antonius Van de Ven},
{\it Compact complex surfaces} (2nd enlarged edition), Ergebnisse der Math. {\bf 4}, Springer, 2004.

\bibitem{BS}
Mauro~C. Beltrametti and Andrew~J. Sommese, \emph{The adjunction theory of complex projective varieties}, Exp.\ Math.\ vol.~16, de Gruyter, Berlin, 1995.

%\bibitem {BLL}  {Aldo Biancofiore, Antonio Lanteri, and Laura\ E.\ Livorni} {\it Ample and spanned vector bundles
%of sectional genera three},
%Math.\ Ann.\ {\bf 291} (1991), 87--101.

\bibitem {BG}  {Spencer Bloch and David Gieseker} {\it The positivity of the Chern classes
of an ample vector bundle},
 Invent.\ Math.\ {\bf 12} (1971), 112--117.

\bibitem{BOSS} {Robert Braun, Giorgio Ottaviani, Michael Schneider, and Frank-Olaf Schreyer}, {\it Classification of conic
bundles in $\mathbb P_5$}, Ann.\ Scuola Norm.\ Sup.\ Pisa Cl.\ Sci.\ {\bf23}(1) (1996), 69--97.


\bibitem{Cat}
Fabrizio Catanese, \emph{On a problem of Chisini}, Duke Math.\ J.\
\textbf{53} (1986), 33--42.

\bibitem{El}
Philippe Ellia, \emph{Chern classes of rank two globally generated vector bundles on $\mathbb P^2$},
Rend.\ Lincei Mat.\ Appl.\  \textbf{24} (2013), 147--163.

\bibitem{FI}
Yoshiaki Fukuma and Hironobu Ishihara, \emph{Complex manifolds polarized by an ample and spanned
line bundle of sectional genus three},
Arch.\ Math.\ \textbf{71} (1998), 159--168.

\bibitem{FPV}
Daniele Faenzi, Francesco Polizzi, and Jean Vall\`es, \emph{Triple planes with $p_g=q=0$},
Trans.\ Amer.\ Math.\ Soc.\  \textbf{371} (2019), 589--639.

\bibitem{libroFujita}
Takao Fujita, \emph{Classification theories of polarized
varieties}, London Math.\ Soc.\ Lecture Notes Series {\bf 155},
Cambridge Univ.\ Press, Cambridge, 1990.

\bibitem {Fu0} {Takao Fujita}, {\it Triple covers by smooth manifolds}, J. Fac.\ Sci.\ Univ.\ Tokyo, Sect.\ IA Math.
{\bf 35}, 1 (1988), 169--175.

\bibitem{Fu2}
Takao~Fujita, \emph{On adjoint bundles of ample vector bundles}. In
``Proc.\ of the Conference in Algebraic Geometry'', Bayreuth, 1990,
Lect.\ Notes in Math.\ \textbf{1507},  Springer, 1992, pp. 105--112.

\bibitem {Ha} {Robin Hartshorne}, {\it Algebraic Geometry},
Springer-Verlag, 1977.

\bibitem {HK} {Jack Huizenga and John Kopper}, {\it Ample stable vector bundles on rational surfaces},
Comm.\ Algebra {\bf 50}, 9 (2022), 3744--3760.

\bibitem{LL} {Antonio Lanteri and Elvira\ L.\ Livorni}
{\it Triple solids with sectional genaus three},
Forum Math.\ {\bf 2} (1990), 297--306.

\bibitem {LM} {Antonio Lanteri and Hidetoshi Maeda},
{\it Adjoint bundles of ample and spanned vector bundles on algebraic surfaces},
 J.\ reine angew.\ Math.\ {\bf 433} (1992),  181--199.

\bibitem{LN3}
Antonio Lanteri and Carla Novelli, \emph{Ample vector bundles of small $\Delta$-genera}, Journal of Algebra \textbf{323} (2010), 671--697.

\bibitem   {LS} {Antonio Lanteri and Andrew\ J.\  Sommese}
{\it A vector bundle characterization of $\mathbb P^n$},
Abh.\ Math.\ Sem.\ Univ.\ Hamburg {\bf 58} (1988), 89--96.

\bibitem{LPS}
Antonio Lanteri, Marino Palleschi, and Andrew J. Sommese,
\emph{Very ampleness of $K_X\otimes {\mathcal L}^{\dim X}$ for
ample and spanned line bundles $\mathcal L$}, Osaka J.\ Math.\
\textbf{26} (1989), 647--664.

\bibitem{Laz1}
Robert Lazarsfeld, \emph{A Barth--type theorem for branched
coverings of projective space}, Math.\ Ann.\ \textbf{249} (1980),
153--162.

\bibitem{Mi3ple}
Rick Miranda, \emph{Triple covers in algebraic geometry}, Amer.\ J.\
Math.\ \textbf{107} (1985), 1123--1158.

\bibitem{No}
Atsushi Noma, \emph{Classification of rank-$2$ ample and spanned vector bundles on surfaces
whose zero loci consist of general points}, Trans.\ Amer.\  Math.\ Soc.\ \textbf{342} (1994), 867--894.

\bibitem{OSS}
Christian Okonek, Michael Schneider, and Heinz Spindler,
\emph{Vector bundles on complex projective spaces}, Birkh$\rm\ddot{a}$user, 1980.

\bibitem{R} Igor Reider,  \emph{Vector bundles of rank 2 and linear systems on algebraic surfaces},
Ann. of Math. {\bf 127}, 2 (1988), 309--316.

\bibitem   {So} {Andrew\ J.\ Sommese}, {\it On manifolds that cannot be ample divisors},
Math.\ Ann.\ {\bf 221} (1976), 55--72.

\bibitem{SW}
Micha{\l} Szurek and Jaros{\l}aw A. Wi{\'s}niewski, \emph{Fano
bundles of rank {$2$} on surfaces}, Compos.\ Math.\ \textbf{76}
(1990), 295--305.

\bibitem{T}
{Sheng-Li Tan}, {\it Triple covers on smooth algebraic varieties}, In ``Geometry and nonlinear partial differential equations''
(Hangzhou, 2001), AMS/IP Stud.\ Adv.\ Math.\ vol. 29, Amer.\ Math.\ Soc., Providence, RI, 2002, pp.\ 143--164.

\end{thebibliography}
\end{document}